\documentclass[11pt, a4paper, dvipsnames]{amsart}

\usepackage[utf8]{inputenc}
\usepackage{lmodern}
\usepackage[T1]{fontenc}
\usepackage[english]{babel}
\usepackage{fullpage}

\usepackage{latexsym}
\usepackage{amsmath,amssymb,amsthm}
\usepackage{mathrsfs}
\usepackage{tikz-cd}

\usepackage{multirow}
\usepackage{booktabs}

\usepackage{color}
\usepackage{hyperref}
\hypersetup{
	colorlinks=true,
	linkcolor=blue,
	citecolor=blue,
	urlcolor=blue
}
\usepackage{xcolor}

\usepackage{soul} 

\usepackage{lineno}

\usepackage[textwidth=20mm]{todonotes}  

\numberwithin{equation}{section} 

\newtheorem{theoremAlph}{Theorem}
\newtheorem{corollaryAlph}[theoremAlph]{Corollary} 

\newtheorem{theorem}{Theorem}[section]
\newtheorem{lemma}[theorem]{Lemma}	
\newtheorem{proposition}[theorem]{Proposition}
\newtheorem{corollary}[theorem]{Corollary}
\theoremstyle{definition}
\newtheorem{definition}[theorem]{Definition} 
\newtheorem{remark}[theorem]{Remark}	
\newtheorem{example}[theorem]{Example}

\theoremstyle{definition} 
\newtheorem*{ack}{Acknowledgements}
\newtheorem*{conflict}{Competing interests}
\newtheorem*{funding}{Financial support}

\newcommand{\C}{\mathbb{C}}
\newcommand{\R}{\mathbb{R}}

\newcommand{\Z}{\mathbb{Z}}
\newcommand{\N}{\mathbb{N}}
\newcommand{\CIso}[1]{[S^1,#1]^{fr}}
\newcommand{\B}{\mathrm{B}}
\newcommand{\E}{\mathrm{E}}
\newcommand{\ttimes}{\mathrel{\widetilde{\times} }}

\newcommand{\bigslant}[2]{{\raisebox{.2em}{$#1$}\left/\raisebox{-.2em}{$#2$}\right.}}

\DeclareFontFamily{U}{mathx}{\hyphenchar\font45}
\DeclareFontShape{U}{mathx}{m}{n}{
      <5> <6> <7> <8> <9> <10>
      <10.95> <12> <14.4> <17.28> <20.74> <24.88>
      mathx10
      }{}
\DeclareSymbolFont{mathx}{U}{mathx}{m}{n}
\DeclareFontSubstitution{U}{mathx}{m}{n}
\DeclareMathSymbol{\bigtimes}{1}{mathx}{"91}

\begin{document}
	
	\title[Free torus actions and twisted suspensions]{Free torus actions and twisted suspensions}
	
	
	\author[F.~Galaz-Garc\'ia]{Fernando Galaz-Garc\'ia}
	\address[Galaz-Garc\'ia]{Department of Mathematical Sciences, Durham University, United Kingdom.}
	\email{\href{mailto:fernando.galaz-garcia@durham.ac.uk}{fernando.galaz-garcia@durham.ac.uk}}

	
	\author[P.~Reiser]{Philipp Reiser}
	\address[Reiser]{Department of Mathematics, University of Fribourg, Switzerland.}
	\email{\href{mailto:philipp.reiser@unifr.ch}{philipp.reiser@unifr.ch}}
	
	
	
	\date{\today}



	
	\begin{abstract}

	We express the total space of a principal circle bundle over a connected sum of two manifolds in terms of the total spaces of circle bundles over each summand, provided certain conditions hold. We then apply this result to provide sufficient conditions for the existence of free circle and torus actions on connected sums of products of spheres and obtain a topological classification of closed, simply-connected manifolds with a free cohomogeneity-four torus action. As a corollary, we obtain infinitely-many manifolds with Riemannian metrics of positive Ricci curvature and isometric torus actions. 
	\end{abstract}
	
	\maketitle


	\section{Introduction and main results}
    Manifolds equipped with torus actions are a central object of study in geometry and topology  (see e.g.\ \cite{E18,FOT08,Gr02,Ko,OR67,Ra68,Se23,W07} and the references therein, to name but a few general references in the literature). Despite being extensively studied, basic questions on these spaces remain open, such as which smooth manifolds admit a smooth, effective torus action. This article addresses this question in the case of free actions.

    If a closed (i.e.\ compact and without boundary) smooth manifold $M$ admits a free smooth torus action, then it is well-known that the Euler characteristic $\chi(M)$ and all Stiefel--Whitney and Pontryagin numbers (provided $M$ is orientable) of $M$ vanish (see Lemmas \ref{L:FREE_ACTION_CHAR} and \ref{L:PR_BDL_CHAR_NR} below). Other topological obstructions can be obtained in certain special cases using spectral sequences (see e.g.\ \cite{PSS10}) or assumptions on the rational homotopy groups of $M$ (see e.g. \cite{GGKR21}), and topological classifications of manifolds with free circle actions in low dimensions were obtained in \cite{DL05,GL71,J14}; see also \cite{ChLa74,HKS20,Kollar,Le73} for classification and obstruction results for almost-free and semi-free torus actions. In this article, we provide sufficient conditions for the existence of smooth, free circle and torus actions on closed, simply-connected manifolds (see Theorems~\ref{T:MAIN}--\ref{T:S1_BUNDLES}, Corollary~\ref{C:FREE_TORUS}, and Theorem~\ref{T:FREE_CIRCLE} below). 
 
    The main application we consider are connected sums of products of spheres. In particular, we show that closed, simply-connected smooth $n$-manifolds with a smooth, free action of $T^{n-4}$ are diffeomorphic to connected sums of products of spheres or non-trivial sphere bundles over $S^2$ (see Theorem~\ref{T:COHOM_4}). These manifolds are known to carry Riemannian metrics of positive Ricci curvature (see remarks before  Corollary~\ref{C:RIC>0}). By exhibiting these manifolds as total spaces of principal torus bundles, we may show that they admit, in fact, Riemannian metrics of positive Ricci curvature which are invariant under the given free torus action (see Corollary~\ref{C:RIC>0}; cf.~\cite{CG20}). Manifolds with such metrics play a role in the study of moduli spaces of Riemannian metrics with positive Ricci curvature (see, for example, \cite{DKT18,G23,KS93,TW15,WZ90}). 

    An important tool we will use are the \emph{twisted suspensions} $\Sigma_e M$ and $\widetilde{\Sigma}_e M$ of a smooth $n$-dimensional manifold $M$ determined by a class $e\in H^2(M;\Z)$. These twisted suspensions, which we will define in Section \ref{S:TW_SUSP}, are obtained by surgery along a fiber of the principal circle bundle over $M$ with Euler class $e$ and generalize the suspensions Duan introduced in \cite{Duan}. These are based on the spinning operation for knots, which is due to Artin \cite{Ar25}.

     Our first main result characterizes certain principal circle bundles in terms of twisted suspensions. Recall that, for $n$-manifolds $M_1$ and $M_2$, we have an isomorphism $H^2(M_1\# M_2;\Z)\cong H^2(M_1;\Z)\oplus H^2(M_2;\Z)$ if $n\geq 4$. A non-trivial integral cohomology class is \emph{primitive} if it is not a multiple of another class. We will denote diffeomorphism between smooth manifolds by the symbol ``$\cong$'' and assume that all manifolds and actions are smooth.

    \begin{theoremAlph}
        \label{T:MAIN}
        Let $B_1$, $B_2$ be closed, oriented $n$-manifolds with $n\geq 5$ and let $P\xrightarrow{\pi}B_1\# B_2$ be a principal $S^1$-bundle. For $i=1,2$, denote by $e_i\in H^2(B_i)$ the restriction of the Euler class of $P$ to $B_i$ and by $P_i\xrightarrow{\pi_i}B_i$ the principal $S^1$-bundle with Euler class $e_i$. If the fiber inclusion in $P_1$ is null-homotopic, or, equivalently, the pull-back of $e_1$ to the universal cover $\widetilde{B}_1$ of $B_1$ is primitive, then $P$ is diffeomorphic to
        \[ P\cong \begin{cases}
            P_1\#\Sigma_{e_2}B_2,\quad & \text{if }\widetilde{B}_1\text{ is non-spin},\\
            P_1\#\widetilde{\Sigma}_{e_2}B_2,\quad &\text{if }\widetilde{B}_1\text{ is spin.}
        \end{cases}
        \]
    \end{theoremAlph}
    Theorem \ref{T:MAIN} generalizes \cite[Theorem B]{Duan}, where the same conclusion is obtained for $B_1$ simply-connected and $e_2=0$.
    
    To apply Theorem~\ref{T:MAIN}, we determine the twisted suspensions of certain manifolds in the following theorem. We denote by $S^2\ttimes S^{n-2}$ the total space of the unique non-trivial linear $S^{n-2}$-bundle over $S^2$.
    Recall that the \emph{divisibility} $d$ of an element $y$ in a free abelian group $G$ is the largest $d\in\N$ such that there exists an element $x\in G$ with $y=dx$. Note that the primitive elements of $G$ are precisely the elements of divisibility $1$.

    \begin{theoremAlph} \label{T:SUSP_EX} We have the following:
        \begin{enumerate}
            \item Let $B$ be a closed, oriented $n$-manifold with $n\geq 5$ and let $P\xrightarrow{\pi}B$ be a principal $S^1$-bundle with Euler class $e\in H^2(B)$. If the fiber inclusion in $P$ is null-homotopic, or, equivalently, the pull-back of $e$ to the universal cover $\widetilde{B}$ is primitive, then
            \[\Sigma_{e}B\cong\begin{cases}
            P\#(S^2\times S^{n-1}),\quad& \text{if }\widetilde{B}\text{ is non-spin},\\
            P\#(S^2\ttimes  S^{n-1}),\quad &\text{if }\widetilde{B}\text{ is spin},
            \end{cases} \]
            and
            \[\widetilde{\Sigma}_{e}B\cong\begin{cases}
            P\#(S^2\ttimes  S^{n-1}),\quad& \text{if }\widetilde{B}\text{ is non-spin},\\
            P\#(S^2\times S^{n-1}),\quad &\text{if }\widetilde{B}\text{ is spin}.
            \end{cases} \]
            \item Let $B=S^k\times S^{n-k}$ with $2\leq k\leq n-2$. Then
            \[ \Sigma_0 B\cong\widetilde{\Sigma}_0 B\cong (S^k\times S^{n-k+1})\# (S^{k+1}\times S^{n-k}). \]
            \item Let $B=S^2\times  S^{n-2}$ or $S^2\ttimes S^{n-2}$, let $e\in H^2(B)$ and let $d$ be the divisibilty of $e$. Then
            \[ \Sigma_e B\cong\begin{cases}
                (S^2\times  S^{n-1})\# (S^3\times S^{n-2}),\quad &\text{if }B=S^2\times S^{n-2}\text{ and }d\text{ is even, or}\\
                & B=S^2\ttimes S^{n-2}\text{ and } d\text{ is odd}, \\
                (S^2\ttimes  S^{n-1})\# (S^3\times S^{n-2}),\quad &\text{else,}
            \end{cases}
            \]
            and
            \[
                \widetilde{\Sigma}_e B\cong\begin{cases}
                    (S^2\times S^{n-1})\# (S^3\times S^{n-1}),\quad & B=S^2\times S^{n-2},\\
                    (S^2\ttimes  S^{n-1})\# (S^3\times S^{n-2}),\quad & B=S^2\ttimes S^{n-2}.
                \end{cases}
                \]
        \end{enumerate}
    \end{theoremAlph}
    We note that item (2) of Theorem \ref{T:SUSP_EX} recovers \cite[Proposition 3.2]{Duan} and extends \cite[Lemma 1.3]{Su90}.
    
    We will say that a manifold $M^n$ \emph{is of the form \eqref{EQ:COND}} if
    \begin{equation}
        \label{EQ:COND}
        M\cong B_1\#\dots\# B_l\quad \text{ for }B_i=S^{m_i}\times S^{n-m_i}\text{ or }B_i= S^2\ttimes  S^{n-2} \tag{$*$}
    \end{equation}
    with $2\leq m_i\leq n-2$ and $n\geq 5$, where we define $M=S^n$ for $l=0$. Note that the diffeomorphism type of a manifold of the form \eqref{EQ:COND} is uniquely determined by its dimension $n$, the Betti numbers $b_2(M),\dots,b_{\lfloor\frac{n}{2}\rfloor}(M)$ (since $b_i(M)=b_{n-i}(M)$ by Poincaré duality), and whether $M$ is spin or not, since $S^2\ttimes S^{n-2}$ is non-spin and
    \[(S^2\ttimes S^{n-2})\#(S^2\ttimes S^{n-2})\cong (S^2\ttimes S^{n-2})\#(S^2\times S^{n-2})\]
    by Corollary \ref{C:CONN_SUM_DIFF} below.

Using Theorems \ref{T:MAIN} and \ref{T:SUSP_EX} and the existence of certain self-diffeomorphisms on connected sums of manifolds of the form \eqref{EQ:COND} with a given simply-connected manifold, we can determine the total space of a principal $S^1$-bundle over manifolds of the form \eqref{EQ:COND}, provided the Euler class is primitive.     

    \begin{theoremAlph}
        \label{T:S1_BUNDLES}
        Let $P\xrightarrow{\pi}B^n$ be a principal $S^1$-bundle with primitive Euler class $e$ and assume that $B$ is of the form \eqref{EQ:COND}. Then $P$ is also of the form \eqref{EQ:COND} with
        \[ b_i(P)=\begin{cases}
            b_2(B)-1,\quad & i=2,n-2,\\
            b_{i-1}(B)+b_{i}(B),\quad & 2<i<n-2.
        \end{cases} \]
        Moreover, $P$ is spin if and only if either $B$ has no $(S^2\ttimes S^{n-2})$-summand, or the restriction of $e$ to each $(S^2\ttimes S^{n-2})$-summand in $B$ has odd divisibility and the restriction of $e$ to each $(S^2\times S^{n-2})$-summand in $B$ has even divisibility.
    \end{theoremAlph}

    We now give several applications of Theorems \ref{T:MAIN}--\ref{T:S1_BUNDLES}. For a topological space $X$ whose first $i\geq 0$ Betti numbers are finite, denote by $\chi_i(X)$ the $i$-th \emph{Euler characteristic}, defined by
    \[ \chi_i(X)=\sum_{j=0}^i(-1)^j b_j(X).\]
    Iterating this definition, we define $\chi^{(0)}_i(X)=(-1)^i b_i(X)$ and, for $m\in\N$,
    \[\chi_i^{(m)}(X)=\sum_{j=0}^i\chi_j^{(m-1)}(X). \]
    We then have $\chi_i(X)=\chi_i^{(1)}(X)$ and $\chi(X)=\chi_n(X)$ if $b_i(X)=0$ for all $i>n$.
    \begin{corollaryAlph}
        \label{C:FREE_TORUS}
        Let $M^n$ be a closed, simply-connected manifold and let $0\leq k\leq n$. Then $M$ admits a free action of the torus $T^k$ with quotient of the form \eqref{EQ:COND} if and only if $M$ is of the form \eqref{EQ:COND} and, for all $1\leq m\leq k$, we have 
        \begin{enumerate}
            \item $(-1)^i\chi_i^{(m)}(M)\geq 0$ for all $i=2,\dots,\lfloor\frac{n-m}{2}\rfloor$,
            \item $\chi_{\frac{n-m}{2}}^{(m)}(M)$ is even if $n-m$ is even, and
            \item $\chi^{(m)}_{n}(M)=0$.
        \end{enumerate}
    \end{corollaryAlph}

    By restricting to the case of $S^1$-actions, we can give further sufficient conditions for the existence of a free action.
    \begin{theoremAlph}
        \label{T:FREE_CIRCLE}
        \begin{enumerate}
            \item Let $M^n$ be of the form \eqref{EQ:COND} and suppose that $n$ is odd. Then there exists $m_0\in\N_0$ such that the manifolds
            \[ M\#_m(S^2\times S^{n-2})\text{ and }M\#_m(S^2\ttimes S^{n-2}) \]
            both admit a free circle action for all $m\geq m_0$.
            \item Let $M^n$ be of the form \eqref{EQ:COND} with $5\leq n\leq 10$ and suppose that $\chi(M)=0$ if $n$ is even and $\chi_4(M)\geq 0$ if $n=9$. Then $M$ admits a free circle action.
        \end{enumerate}
    \end{theoremAlph}

    The simplest examples not covered by Theorem \ref{T:FREE_CIRCLE} with vanishing Euler characteristic are the manifolds $\#_m(S^3\times S^6)$ with $m\geq 2$. By Proposition \ref{P:S3xS6} below, these manifolds do not admit a free circle action when $m$ is odd. To the best of our knowledge, it is open whether these manifolds admit a free circle action when $m$ is even.


    We can also use Theorem \ref{T:S1_BUNDLES} to determine the total space of a principal torus bundle over any closed, simply-connected $4$-manifold (see Theorem \ref{T:Tk_BUNDLES} below). This yields a complete topological classification of the total spaces of such principal bundles, and extends a result of Duan and Liang \cite{DL05} for principal circle bundles and of Duan \cite{Duan} for principal $T^k$-bundles over $4$-manifolds $M^4$ with $b_2(M)=k$.
 
    We apply this result to free torus actions of large cohomogeneity. Note that the dimension of a torus acting freely on a closed, simply-connected $n$-manifold with $n\geq 4$ must be at most $n-4$ (see Remark \ref{R:COHOM_4} below). In the case of maximal dimension, we have the following classification. For that we first define $a_{ki}(r)$ for $r,k\in\N_0$ and $2\leq i \leq k+2$ by
    \[a_{ki}(r)= (i-2)
    \binom{k}{i-1}
    +r
    \binom{k}{i-2}
    +(2+k-i)
    \binom{k}{i-3}. 
    \]
    \begin{theoremAlph}
        \label{T:COHOM_4}
        A closed, simply-connected $n$-manifold $M$ admits a free action of the torus $T^{n-4}$ if and only if $M$ is of the form \eqref{EQ:COND} with $b_i(M)=a_{n-4,i}(b_2(M))$ for all $2\leq i\leq n-2$.
    \end{theoremAlph}
 
    An interesting special case of Theorem \ref{T:COHOM_4} is where the quotient space $B^4=M/T^{n-4}$ itself admits an effective action of a $2$-torus (see, for example, \cite{CG20,GGK14}). It is then possible to lift the action to $M$, so that, together with the free $T^{n-4}$-action, we obtain a torus action of cohomogeneity two on $M$ (see \cite{HY,Su} and cf.\ \cite{CG20}). Closed, simply-connected manifolds with a cohomogeneity-two torus action have been classified (both topologically and equivariantly) by Orlik and Raymond \cite{OR70} in dimension $4$ and by Oh \cite{Oh82,Oh83} in dimensions $5$ and $6$. The orbit space structure and equivariant classification of closed, simply-connected $n$-manifolds with a cohomogeneity-two torus action may be found in \cite{KMP74}. In dimensions $7$ and above, however, no topological classification is known. By the above lifting argument, in combination with the four-dimensional classification, Theorem \ref{T:COHOM_4} provides a topological classification in any dimension, provided there exists a free cohomogeneity-four subaction. If we instead use Oh's $6$-dimensional classification, we can strengthen this as follows.
    
    \begin{corollaryAlph}
    \label{C:COHOM_2}
        A closed, simply-connected $n$-manifold $M$, $n\geq 6$, admits a smooth effective action of $T^{n-2}$ with a free subaction of a torus  of dimension $(n-6)$ if and only if $M$ is of the form \eqref{EQ:COND} with $b_i(M)=a_{ki}(b_2(M))$ for all $2\leq i\leq n-2$.
    \end{corollaryAlph}

    We note that not all  cohomogeneity-two torus actions on a closed, simply-connected $n$-manifold $M$ have a free subaction as in Corollary \ref{C:COHOM_2} (see Remark~\ref{R:NO.FREE.SUBTORUS} below). However, it is open whether the manifolds in Corollary \ref{C:COHOM_2} already provide all diffeomorphism types of closed, simply-connected manifolds with a cohomogeneity-two torus action. In dimensions $5$ and $6$, this is known to be true if one considers free cohomogeneity-four subactions (see \cite{CG20,Oh82,Oh83}).

    Using the core metric construction introduced in \cite{Bu19}, one obtains that every manifold of the form \eqref{EQ:COND} admits a metric of positive Ricci curvature, by \cite{Bu20} and \cite{Re21}. However, these metrics need not be invariant under the actions established in Corollaries \ref{C:FREE_TORUS} and \ref{C:COHOM_2} and Theorems \ref{T:FREE_CIRCLE} and \ref{T:COHOM_4}. The existence of an invariant metric of positive Ricci curvature can now be obtained in combination with the lifting results of \cite{GPT98}.

    \begin{corollaryAlph}
        \label{C:RIC>0}
        Let $M$ be a manifold of the form \eqref{EQ:COND} satisfying the assumptions of Corollary \ref{C:FREE_TORUS} or \ref{C:COHOM_2}, or Theorem \ref{T:FREE_CIRCLE} or \ref{T:COHOM_4}, thus admitting a free action of a torus. Then $M$ admits a metric of positive Ricci curvature that is invariant under the free torus action.
    \end{corollaryAlph}
  
    The existence of invariant metrics of positive Ricci curvature on the manifolds in Theorem \ref{T:COHOM_4} has already been shown in \cite{CG20} without identifying the total spaces if the dimension of the total space is at least $7$.    

    This article is organized as follows. In Section \ref{S:PRELIMINARIES}, we recall basic facts on principal torus bundles and results from differential topology. In Section \ref{S:FRAMED_CIRCLES}, we study isotopy classes of normally framed circles which will be crucial for the proofs of Theorems \ref{T:MAIN} and \ref{T:SUSP_EX}, and in Section \ref{S:SURGERY} we consider the effect of surgery on a normally framed circle and establish the existence of certain self-diffeomorphisms on manifolds of the form \eqref{EQ:COND}. In Section \ref{S:TW_SUSP}, we introduce the twisted suspensions and prove Theorems \ref{T:MAIN} and \ref{T:SUSP_EX}. Finally, in Section \ref{S:CIRCLE_BUNDLES}, we apply Theorems \ref{T:MAIN} and \ref{T:SUSP_EX} to prove Theorems \ref{T:S1_BUNDLES}, \ref{T:FREE_CIRCLE}, and \ref{T:COHOM_4}, and Corollaries \ref{C:FREE_TORUS}, \ref{C:COHOM_2}, and \ref{C:RIC>0}.

    \begin{ack}
        The authors would like to thank Martin Kerin and Sam Hagh Shenas Noshari for helpful comments on an earlier version of this article, Lee Kennard and Lawrence Mouillé for providing the example in Remark~\ref{R:NO.FREE.SUBTORUS}, and Haibao Duan for helpful discussions. Philipp Reiser would also like to thank the Department of Mathematical Sciences of Durham University for its hospitality during a first visit where this work was initiated and a second visit where it was completed. Finally, the authors would like to thank the anonymous referee for their suggestions that helped to improve the exposition.
    \end{ack}

 
	\section{Preliminaries}
	\label{S:PRELIMINARIES}

	We will identify $\R^k$ with a subspace of $\R^l$ if $k\leq l$ via the map 
    \[
    (v_1,\dots,v_k)\mapsto (v_1,\dots,v_k,0,\dots,0).
    \]
    Similarly, we consider $\mathrm{SO}(k)$ as a subgroup of $\mathrm{SO}(l)$ by applying $\phi\in\mathrm{SO}(k)$ to the first $k$ entries of $v\in\R^l$. We will use homology and cohomology with integer coefficients, unless explicitly stated otherwise. We will denote the fundamental class of a closed, oriented manifold $M$ by $[M]$. The closed $m$-disk will be denoted by $D^m$. The symbol ``$\cong$'' will denote isomorphism between algebraic structures and diffeomorphism between manifolds. Given a vector space $V$ and a manifold $M$, we denote by $\underline{V}_M$ the trivial bundle $M\times V\to M$.
	
    
	\subsection{Auxiliary facts on principal torus bundles} 
    We denote by $T^k$ the torus of dimension $k$, i.e.\ $T^k=S^1\times\overset{k}{\cdots}\times S^1$ and $S^1\subseteq \C$ is the unit circle. We first recall the connection between principal torus bundles and free torus actions.
    \begin{lemma}
        \label{L:FREE_ACTION_CHAR}
        A manifold $M$ admits a free action of a Lie group $G$ if and only if it is the total space of a principal $G$-bundle. In this case, if $G=T^k$, the Euler characteristic $\chi(M)$ vanishes.
    \end{lemma}
    \begin{proof}
        For the first statement see, e.g.\ \cite[Corollary VI.2.5]{Br72}. If $M$ admits an effective $T^k$-action, then the Euler characteristic of $M$ equals the Euler characteristic of the fixed point set of the action (see \cite{Ko58} and cf.\ \cite[Ch.\ II, Theorem 5.5]{Ko}). In particular, if the action is free, then $\chi(M)$ vanishes.
    \end{proof}
  
    Now, let $P\xrightarrow{\pi}X$ be a principal $T^k$-bundle. Let $\E S^1\xrightarrow{\pi_{S^1}}\B S^1$ be the universal bundle for $S^1$ (we refer to \cite[Sections 4.10--4.13]{Hu94} for the definition and basic properties of universal bundles). Then the product bundle
    \[\E S^1\times\overset{k}{\cdots}\times \E S^1\to \B S^1\times\overset{k}{\cdots}\times \B S^1\]
    is the universal bundle for $T^k$, and we denote the corresponding bundle map by $\pi_{T^k}$. Hence, there exists a map $f_\pi\colon X\to \B S^1\times\overset{k}{\cdots}\times \B S^1$ such that $\pi$ is isomorphic to $f_{\pi}^*\pi_{T_k}$.
    Since $f_{\pi}$ is unique up to homotopy, we obtain a unique element in
    \begin{align}
    \label{eq:product.homotopy.classes}
    [X,\B S^1\times\overset{k}{\cdots}\times \B S^1  ]\cong [X,\B S^1]\times \overset{k}{\cdots}\times  [X,\B S^1].
    \end{align}
    Since $\B S^1$ is a $K(\Z,2)$-space, the right-hand side of equation~\eqref{eq:product.homotopy.classes} can be identified with $H^2(X,\Z)\times \overset{k}{\cdots}\times H^2(X,\Z)$. Thus, the bundle $\pi$ is uniquely determined by a $k$-tuple
    \[e(\pi)=(e_1(\pi),\dots,e_k(\pi))\in H^2(X,\Z)\times \overset{k}{\cdots}\times H^2(X,\Z).\]
    We call this $k$-tuple the \emph{Euler class} of $\pi$ and note that it coincides with the usual definition of the Euler class if $k=1$.

    \begin{lemma}
        \label{L:TORUS_BDL_STR}
        Let $P\xrightarrow{\pi}X$ be a principal $T^k$-bundle with Euler class $(e_1(\pi),\dots,e_k(\pi))$. Then there is a sequence of principal $S^1$-bundles $P_i\xrightarrow{\pi_i}P_{i-1}$, $i=1,\dots,k$, such that
        \begin{enumerate}
            \item $P_k=P$, $P_0=X$ and $\pi_1\circ\dots\circ \pi_k=\pi$; 
            \item $e(\pi_i)=\pi_{i-1}^*\dots \pi_1^* e_i(\pi)$.
        \end{enumerate}
    \end{lemma}
    \begin{proof}
        We set $P_i=P/T^{k-i}$, where we view $T^j$, for $j<k$, as a subgroup of $T^k$ via
        \[ T^j\cong (\{1\}\times \overset{k-j}{\cdots}\times \{1\})\times T^j\subseteq T^k. \]
        Then the projection $P_i\xrightarrow{\pi_{i}} P_{i-1}$ is a principal $S^1$-bundle, where the action is induced by the action of the $i$-th $S^1$-factor of $T^k$ on $P$. This proves claim (1).

        For the second claim we show that the projection $P_i\to X$, when viewed as a principal $T^i$-bundle, has Euler class $(e_1(\pi),\dots,e_i(\pi))$. By construction of $P_i$, the bundle $P_i\to X$ is the pull-back along $f_\pi$ of the principal $T^i$-bundle
        \[\E S^1\times \overset{i}{\cdots}\times \E S^1\times \B S^1\times \overset{k-i}{\cdots}\times \B S^1\to \B S^1\times \overset{k}{\cdots}\times \B S^1, \]
        where the bundle map is given by $\pi_{T^i}$ on the first $i$ factors and by the identity on the last $(k-i)$ factors. We obtain the same bundle when we pull back the universal bundle $\pi_{T^i}\colon\E S^1\times\overset{i}{\cdots}\times \E S^1\to\B S^1\times\overset{i}{\cdots}\times \B S^1$ along $\textrm{pr}_i\circ f_\pi$, where $\textrm{pr}_i$ denotes the projection $\B S^1\times \overset{k}{\cdots}\times \B S^1\to\B S^1\times \overset{i}{\cdots}\times \B S^1$ onto the first $i$ factors. Thus, the Euler class of $P_i\to X$ is given by $(e_1(\pi),\dots,e_i(\pi))$.
    \end{proof}
    \begin{lemma}
    	\label{L:PR_BDL_TOPOLOGY}
        Let $P\xrightarrow{\pi}X$ be a principal $T^k$-bundle with Euler class $e(\pi) = (e_1(\pi),\dots,e_k(\pi))$ such that $X$ is simply-connected. Then
        \[\pi_1(P)\cong \bigslant{\Z^k}{\mathrm{im}(e(\pi))} \]
        and
        \[ H^2(P)\cong \bigslant{H^2(X)}{\langle e_1(\pi),\dots,e_k(\pi)\rangle }, \]
        where in the first case we view $e(\pi)$ as a homomorphism $H_2(X)\to \Z^k$ and in the second case the isomorphism is induced by $\pi$.
        In particular, $P$ is simply-connected if and only if the Euler class $e(\pi)$ generates a direct summand in $H^2(X)$, that is, $(e_1(\pi),\dots,e_k(\pi))$ can be extended to a basis of $H^2(X)$.
    \end{lemma}
    \begin{proof}
        The long exact sequence of homotopy groups for the bundles $\pi$ and $\pi_{T^k}$ together with the induced maps of $f_\pi$ gives the following commutative diagram with exact rows (see e.g.\ \cite[17.4 and 17.5]{St99}):
        \[
            \begin{tikzcd}[column sep=scriptsize]
                \pi_2(P)\arrow{r}{\pi_*}\arrow{d}{} & \pi_2(X) \arrow{r}{}\arrow{d}{{f_\pi}_*} & \pi_1(T^k)\arrow{r}{}\arrow{d}{\textrm{id}_{\pi_1(T^k)}} & \pi_1(P)\arrow{r}{\pi_*}\arrow{d}{} & \pi_1(X)\arrow{d}{{f_\pi}_*}\arrow{r}{} & \pi_0(T^k)\arrow[swap]{d}{\textrm{id}_{\pi_0(T^k)}}\\
                \pi_2(\bigtimes_k\E S^1 )\arrow{r}{\pi_*} & \pi_2(\bigtimes_k\B S^1) \arrow{r}{} & \pi_1(T^k)\arrow{r}{} & \pi_1(\bigtimes_k\E S^1)\arrow{r}{\pi_*} & \pi_1(\bigtimes_k\B S^1)\arrow{r}{} & \pi_0(T^k)
            \end{tikzcd}.
        \]
        Since $\E S^1$ is contractible, all its homotopy groups vanish, so the map $\pi_2(\bigtimes_k\B S^1)\to\pi_1(T^k)$ is an isomorphism and $\pi_1(\bigtimes_k\B S^1)$ is trivial. In particular, the group $\pi_2(\bigtimes_k\B S^1)$ is isomorphic to $\Z^k$. Since $X$ is simply-connected, it follows that the group $\pi_1(P)$ is isomorphic to the quotient of $\pi_2(\bigtimes_k\B S^1)$ by the image of ${f_\pi}_*$. By the Hurewicz theorem, we can identify the image of ${f_\pi}_*$ in $\pi_2(\bigtimes_k\B S^1)$ with the image of the induced map of $f_\pi$ in homology, which by construction is precisely the image of the Euler class $(e_1(\pi),\dots,e_k(\pi))$.

        For the cohomology, we first consider the case $k=1$, i.e.\ $\pi$ is a principal $S^1$-bundle, and apply the Gysin sequence (see e.g.\ \cite[Theorem 12.2]{MS74}):
        \[
            \begin{tikzcd}[column sep=large]
                H^0(X)\arrow{r}{\cdot\smile e_1(\pi) } & H^2(X)\arrow{r}{\pi^*} & H^2(P)\arrow{r}{} & H^1(X).
            \end{tikzcd}
        \]
        Since $X$ is simply-connected, we have $H^1(X)=0$, so $\pi^*\colon H^2(X)\to H^2(P)$ is surjective with kernel given by the image of the map $\cdot\smile e_1(\pi)\colon H^0(X)\to H^2(X)$, which is precisely the subgroup generated by $e_1(\pi)$.

        For general $k$, we apply Lemma \ref{L:TORUS_BDL_STR} to divide $\pi$ into a sequence of principal $S^1$-bundles. Repeated application of the above argument for principal $S^1$-bundles now gives the claim.
    \end{proof}
    In case of a non-simply-connected base we have the following result.
    \begin{lemma}
        \label{L:PR_BDL_TOPOLOGY2}
        Let $P\xrightarrow{\pi}X$ be a principal $S^1$-bundle with Euler class $e(\pi)$. Then the inclusion of a fiber in $P$ is null-homotopic if and only if the pull-back of $e(\pi)$ to the universal cover $\widetilde{X}$ is primitive.
    \end{lemma}
    \begin{proof}
        Let $\overline{P}\xrightarrow{\widetilde{\pi}}\widetilde{X}$ denote the pull-back of $\pi$ along the covering projection $\widetilde{X}\to X$. The long exact sequence of homotopy groups for the bundles $\pi$ and $\widetilde{\pi}$ gives the following commutative diagram with exact rows:
        \[
            \begin{tikzcd}
                \pi_2(\widetilde{X})\arrow{r}\arrow{d} & \pi_1(S^1)\arrow{r}\arrow{d}{=} & \pi_1(\overline{P})\arrow{d}\arrow{r}{\widetilde{\pi}} & \pi_1(\widetilde{X})\arrow{d} \\
                \pi_2(X)\arrow{r} & \pi_1(S^1)\arrow{r} & \pi_1(P)\arrow{r}{\pi} & \pi_1(X)
            \end{tikzcd}
        \]
        Since the map $\pi_2(\widetilde{X})\to\pi_2(X)$ is an isomorphism, it follows that the map $\pi_1(S^1)\to\pi_1(P)$ is trivial if and only if the map $\pi_2(\widetilde{X})\to\pi_1(S^1)$ is surjective. Since $\widetilde{X}$ is simply-connected, this is the case if and only if $\pi_1(\overline{P})$ is trivial. Since the Euler class of $\widetilde{\pi}$ is the pull-back of $e(\pi)$ along the projection $\widetilde{X}\to X$, the claim follows from Lemma \ref{L:PR_BDL_TOPOLOGY}.
    \end{proof}
    
	Recall that a \emph{stable characteristic class} is an element $c\in H^i(\mathrm{BO};R)$ for a ring $R$. For a vector bundle $E\xrightarrow{\pi}X$ of rank $k$, we then set $c(\pi)=f_\pi^*\iota_k^*c$, where $\iota_k\colon \mathrm{BO}(k)\to\mathrm{BO}$ is the map induced by the inclusion $\mathrm{O}(k)\hookrightarrow\mathrm{O}$ and $f_\pi\colon X\to\B \mathrm{O}(k)$ is the classifying map of $\pi$. For a manifold $M$ we set $c(M)=c(TM)$. We then have
	\[c(\pi\oplus\underline{\R}_X)=c(\pi) \]
	for every vector bundle $E\xrightarrow{\pi}X$. The Stiefel--Whitney classes $w_i\in H^i(\mathrm{BO};\Z/2)$ and the Pontryagin classes $p_i\in H^{4i}(\mathrm{BO};\Z)$ are examples of stable characteristic classes.
    \begin{lemma}
    	\label{L:PR_BDL_CHAR_CL}
        Let $P\xrightarrow{\pi}M$ be a principal $G$-bundle for a Lie group $G$ and let $c\in H^i(\mathrm{BO};R)$ be a stable characteristic class. Then
        \[c(P)=\pi^*c(M). \]
    \end{lemma}
	\begin{proof}
		The tangent bundle of $P$ is given by
        \[ TP\cong \pi^* TM\oplus T_\pi P, \]
        where $T_\pi P=\ker(\pi_*)$ is the bundle of vertical vectors. This can be seen by choosing a connection on the bundle $P$, so that the horizontal bundle is isomorphic to $\pi^* TM$. The bundle $T_\pi P$ is now isomorphic to the trivial bundle $\underline{\mathfrak{g}}_P$ via the isomorphism
        \[P\times\mathfrak{g}\to T_\pi P,\quad (p,X)\mapsto \frac{d}{dt}\Big|_{t=0}(p\cdot \exp(tX)). \]
        It follows that
        \[c(P)=c(\pi^*TM\oplus \underline{\mathfrak{g}}_P)=\pi^*c(M). \]
	\end{proof}
    \begin{corollary}
        \label{C:P_W2}
    	Let $P\xrightarrow{\pi}M$ be a principal $T^k$-bundle with Euler class $(e_1(\pi),\dots,e_k(\pi))$ and assume that $M$ is orientable. Then $P$ is spin if and only if
    	\[w_2(M)\in\langle e_1(\pi),\dots,e_k(\pi)\rangle\mod 2. \]
    \end{corollary}
	\begin{proof}
		Since $M$ is orientable, it follows from Lemma \ref{L:PR_BDL_CHAR_CL} that
		\[w_1(P)=\pi^* w_1(M)=0. \]
		Hence, $P$ is orientable. For the second Stiefel--Whitney class $w_2(P)$ we have $w_2(P)=\pi^*w_2(M)$ by Lemma \ref{L:PR_BDL_CHAR_CL} and by Lemma \ref{L:TORUS_BDL_STR} and the Gysin sequence in $\Z/2$-coefficients (cf.\ \cite[Section 2.2]{Re22}) that the kernel of $\pi^*\colon H^2(M;\Z/2)\to H^2(P;\Z/2)$ is given by
		\[\langle e_1(\pi),\dots,e_k(\pi)\rangle\mod 2. \]
	\end{proof}
    Lemma \ref{L:PR_BDL_CHAR_CL} also provides a topological obstruction for the existence of free torus actions.
    \begin{corollary}
        \label{L:PR_BDL_CHAR_NR}
        Let $M$ be a closed $n$-manifold that admits a free $T^k$-action. Then any product of stable characteristic classes of $M$ of total degree at least $n-k+1$ vanishes. In particular, all Stiefel--Whitney numbers of $M$ and (if $M$ is orientable) all Pontryagin numbers of $M$ vanish.
    \end{corollary}
    \begin{proof}
        By Lemma \ref{L:FREE_ACTION_CHAR}, $M$ is the total space of a principal $T^k$-bundle $M\xrightarrow{\pi} B$. By Lemma \ref{L:PR_BDL_CHAR_CL}, since $B$ has dimension $n-k$, any cup product of stable characteristic classes with total degree at least $n-k+1$ vanishes.
    \end{proof}
	
	
	\subsection{Auxiliary results on smooth manifolds and vector bundles}
	Recall that the \textit{normal bundle }$\nu_N$ of an embedded submanifold $ N\subseteq M$ is the bundle
	\[\nu_N=\bigslant{TM|_N}{TN}. \]
	
	By choosing a Riemannian metric on $M$, we can identify $\nu_N$ with the orthogonal complement of $TN$ within $TM|_N$.
	
	We will use the following relative version of the classical Whitney embedding theorem.
	\begin{theorem}[{Relative Weak Whitney Embedding Theorem, see \cite[Theorem 5]{Wh36}}]
		\label{T:W-EMBEDDING}
		Let $f\colon N\to M$ be a continuous map and let $A\subseteq N$ be a closed subset such that $f|_A\colon A\to M$ is a smooth embedding. If $\dim(M)> 2\dim(N)$, then there is an embedding $g\colon N\hookrightarrow M$ which is homotopic to $f$ such that $g|_A=f|_A$.
	\end{theorem}
The preceding theorem implies the following result, which is also due to Whitney \cite{Wh36}.
	\begin{theorem}[{\cite[Theorem 6]{Wh36}}]
		\label{T:HOM-ISO}
		Let $f_0,f_1\colon N\to M$ be smooth maps that are homotopic. If $\dim(M)>2\dim(N)+1$, then $f_0$ and $f_1$ are isotopic.
	\end{theorem}
	
	The following results are well-known. We include proofs for completeness.

	\begin{proposition}
		\label{P:BUNDLES_SURFACES}
		\begin{enumerate}
		    \item A vector bundle $E\xrightarrow{\pi} S^1$ is trivial if and only if the first Stiefel-Whitney class $w_1(\pi)$ vanishes, that is, if and only if the bundle $\pi$ is orientable.
		    \item Let $E\xrightarrow{\pi} S$ be an orientable vector bundle of rank $k\geq3$ over a closed surface $S$. Then $\pi$ is trivial if and only if the second Stiefel--Whitney class $w_2(\pi)$ vanishes.
		\end{enumerate}
	\end{proposition}
	\begin{proof}
	    \begin{enumerate}
	        \item Assume that $w_1(\pi)=0$. Then the bundle $\pi$ is orientable, hence its classifying map $f_\pi\colon S^1\to \mathrm{BO}(k)$ lifts to a map
	        \[\tilde{f}_\pi\colon S^1\to\mathrm{BSO}(k). \]
	        Since $\mathrm{BSO}(k)$ is simply-connected, this map is null-homotopic, so $\pi$ is trivial.
	        \item Assume that $w_2(\pi)=0$. Then the bundle $\pi$ admits a spin structure, hence its classifying map $f_\pi\colon S\to \mathrm{BSO}(k)$ lifts to a map
    		\[\tilde{f}_\pi\colon S\to\mathrm{BSpin}(k). \]
    		Since $k\geq3$, the group $\mathrm{Spin}(k)$ is simply-connected (see e.g.\ \cite[Theorem I.2.10]{LM89}). Hence, the space $\mathrm{BSpin}(k)$ is 2-connected. By obstruction theory, it follows that the map $\tilde{f}_\pi$, and hence $f_\pi$, is null-homotopic (see e.g.\ \cite[Corollary 7.13]{DK01}), so the bundle $\pi$ is trivial.
	    \end{enumerate}
	\end{proof}

    \begin{lemma}
        \label{L:SPIN}
        The complex projective space $\C P^n$ is spin if and only if $n$ is odd. Further, for every $n\in\N$ with $n\geq 2$ there exists a unique non-trivial linear sphere bundle over $S^2$, whose total space, denoted by $S^2\ttimes S^n$, is non-spin.
    \end{lemma}
    \begin{proof}
        For the first statement we have $w_2(\C P^n)=c_1(\C P^n)\mod 2=n+1\mod 2$, see e.g.\ \cite[Theorem 14.10]{MS74}. For the second statement, since $\pi_1(\mathrm{SO}(n+1))\cong\Z /2$, there exists a unique non-trivial vector bundle $E\xrightarrow{\xi}S^2$ of rank $(n+1)$. Hence, there exists a unique non-trivial linear $S^n$-bundle $S(E)\xrightarrow{\pi} S^2$. By Proposition \ref{P:BUNDLES_SURFACES}, $w_2(\xi)\neq 0$. By choosing a horizontal distribution for the bundle $\xi$, which is isomorphic to $\pi^*TS^2$, we have (cf.\ Lemma \ref{L:PR_BDL_CHAR_CL})
        \[TS(E)\oplus\underline{\R}_{S(E)}\cong \pi^*TS^2\oplus\pi^*E.  \]
        Hence,
        \[w_2(S(E))=w_2(S(E)\oplus\underline{\R}_{S(E)})=\pi^*w_2(S^2)+\pi^*w_2(\pi)=\pi^*w_2(E). \]
        By the Gysin sequence, the map $H^2(S^2)\xrightarrow{\pi^*}H^2(S(E))$ is injective, hence $w_2(S(E))$ is non-trivial.
    \end{proof}

Finally, we recall the following theorem, which is known as the \emph{Disc Theorem of Palais} \cite[Theorem 5.5]{Pa59}.

 \begin{theorem}
 \label{T:DISK-ISOTOPY}	
 	Let $f_1,f_2\colon D^m\to M$ be embeddings. If $m=\dim(M)$ and $M$ is orientable, assume  in addition that both $f_1$ and $f_2$ are orientation preserving. Then $f_1$ and $f_2$ are isotopic.
 \end{theorem}

 
	\section{Normally framed circles}
    \label{S:FRAMED_CIRCLES}

    In this section, $M$ will denote an oriented manifold of dimension $n\geq5$. Some of the results in this section were already obtained by Goldstein and Lininger \cite{GL71} and Duan \cite{Duan} when $M$ is simply-connected. 

    \begin{definition}
        Let $f\colon S^1\hookrightarrow M$ be an embedding. A \emph{normal framing} of $f$ is an orientation-preserving embedding $\varphi\colon S^1\times D^{n-1}\hookrightarrow M$ such that $\varphi(\cdot,0)=f$. We introduce two equivalence relations:
        \begin{enumerate}
            \item Two normal framings $\varphi_0$ and $\varphi_1$ of $f$ are \emph{isotopic} if they are isotopic as embeddings, i.e.\ if there exists a smooth homotopy $\varphi_t$, $t\in [0,1]$, in $M$ between $\varphi_0$ and $\varphi_1$ such that $\varphi_t$ is an embedding for all $t\in[0,1]$. The set of isotopy classes of framings of embeddings $S^1\hookrightarrow M$ is denoted by $\CIso{M}$.
            \item Two normal framings $\varphi_0$ and $\varphi_1$ of $f$ are \emph{equivalent} if they are isotopic through normal framings of $f$.
        \end{enumerate}        
    \end{definition}
    
    Note that normal framings exist for any embedding $f\colon S^1\hookrightarrow M$: The orientation on $M$, together with the standard orientation on $S^1$, induces an orientation on $\nu_{f(S^1)}$ according to the splitting
    \[ TM|_{f(S^1)}\cong Tf(S^1)\oplus \nu_{f(S^1)}. \]
    By Proposition \ref{P:BUNDLES_SURFACES}, it follows that $\nu_{f(S^1)}$ is trivial and hence, by choosing a Riemannian metric on $M$, we obtain an embedding $S^1\times D^{n-1}\hookrightarrow M$ via the exponential map.
    
    It is clear from the definition that equivalent normal framings are isotopic. As we will see below, the converse holds in some cases, but not in general.

    \begin{lemma}
        \label{L:UNIQUENESS_TUB_NBHD}
        Let $f\colon S^1\hookrightarrow M$ be an embedding and let $\varphi_0$, $\varphi_1$ be normal framings of $f$. Then there exists a normal framing $\varphi_1'$ of $f$ that is equivalent to $\varphi_1$ and a smooth map $\alpha\colon S^1\to \mathrm{SO}(n-1)$ such that
        \[ \varphi_0(\lambda,\alpha_\lambda v)=\varphi_1'(\lambda,v) \]
        for all $(\lambda,v)\in S^1\times D^{n-1}$. In particular, there exist exactly two equivalence classes of normal framings of $f$.
    \end{lemma}
    \begin{proof}
        The first statement follows directly from the uniqueness of tubular neighborhoods (see e.g.\ \cite[Corollary III.3.2]{Ko}) and the second statement then follows from the fact that $\pi_1(\mathrm{SO}(n-1))\cong \Z/2$, as $n\geq5$.
    \end{proof}
    It follows from Lemma~\ref{L:UNIQUENESS_TUB_NBHD} that, for an embedding $f\colon S^1\hookrightarrow M$, there are at most two isotopy classes of normal framings. To analyze when we have equality, we introduce the following notion.

    Let $f\colon S^1\hookrightarrow M$ be an embedding and let $F\colon T^2\hookrightarrow M$ be an embedding with $F(\cdot,1)=f$. We view $F$ as a self-isotopy $F_t$ of $f$ via $F_t=F(\cdot,e^{2\pi i t})$. Given a normal framing $\varphi$ of $f$, we extend $\varphi$ along the isotopy $F_t$ to an isotopy $\varphi_t$. We then define $F_*[\varphi]$ as the equivalence class of normal framings of $f$ represented by $\varphi_1$.
    
    \begin{lemma}
        \label{L:F_*_WELL_DEF}
        The class $F_*[\varphi]$ is well-defined, i.e.\ it does not depend on the choice of extension $\varphi_t$.
    \end{lemma}
    \begin{proof}
        Let $\varphi_t$ and $\psi_t$ be extensions of $\varphi$ along $F$. Then the maps $[0,1]\times S^1\times D^{n-1}\to [0,1]\times M$,
        \[(t,\lambda,v)\mapsto (t,\varphi_t(\lambda,v)),\quad (t,\lambda,v)\mapsto (t,\psi_t(\lambda,v)), \]
        are neat tubular neighborhoods of the neat submanifold $N=\{(t,F_t(\lambda))\mid (t,\lambda)\in[0,1]\times S^1\}$ of $[0,1]\times M$ in the sense of \cite[Chapter III.4]{Ko}. By the uniqueness of neat tubular neighborhoods (see e.g.\ \cite[Theorem III.4.2 and subsequent remark]{Ko}), after applying an isotopy of neat tubular neighborhoods that fixes $N$ pointwise (which corresponds to isotopies of $\psi_0$ and $\psi_1$ on the boundary components that fix $f(S^1)$ pointwise), we can assume that
        \[ \psi_t(\lambda,\cdot)=\varphi_t(\lambda,\alpha_{(t,\lambda)}(\cdot))\]
        for a smooth map $\alpha\colon [0,1]\times S^1\to \mathrm{SO}(n-1)$ with $\alpha_{(0,\cdot)}$ homotopic to the constant map $\equiv\mathrm{id}$. This shows that $\psi_1$ is equivalent to $\varphi_1$, where the isotopy is given by $\varphi_1(\lambda,\alpha_{(t,\lambda)}(\cdot))$.
    \end{proof}
    \begin{lemma}
        \label{L:F*_W2}
        Let $f\colon S^1\hookrightarrow M$ be an embedding and let $F\colon T^2\hookrightarrow M$ be an embedding with $F(\cdot,1)=f$. Then, for any normal framing $\varphi$ of $f$, we have $[\varphi]=F_*[\varphi]$ if and only if
        \[w_2(M)\frown F_*[T^2]_{\Z/2}=0.\]
    \end{lemma}
    \begin{proof}
        Let $\xi=F_*[T^2]_{\Z/2}\in H_2(M;\Z/2)$. By Proposition \ref{P:BUNDLES_SURFACES}, the normal bundle $\nu_F$ is trivial if and only if $w_2(\nu_F)=0$. We have
		\[F^*w_2(\nu_F)=F^*w_2(\nu_F\oplus TF(T^2))=F^*w_2(TM|_{F(T^2)})=w_2(F^*TM)=F^*w_2(M),\]
		which vanishes if and only if
		\[0=F^*w_2(M)\frown [T^2]_{\Z/2}=w_2(M) \frown F_*[T^2]_{\Z/2}=w_2(M) \frown \xi. \]
		Hence, $w_2(M)\frown \xi=0$ if and only if $\nu_F$ is trivial.
		
		Now, suppose that $\nu_F$ is trivial, i.e.\ there exists an embedding $\bar{F}\colon T^2\times D^{n-2}\hookrightarrow M$ such that $\bar{F}(\cdot,0)=F$. The map $\varphi_t\colon S^1 \times D^{n-1}\to M$,
		\[\varphi_t(\lambda,v)=\bar{F}((\lambda,e^{2\pi i(t+\frac{v_{n-1}}{4})}),(v_1,\dots,v_{n-2})) \]
		for $\lambda\in S^1$, $v=(v_1,\dots,v_{n-1})\in D^{n-1}$ is an isotopy along $F_t$ with $\varphi_0=\varphi_1$, showing that $F$ induces the identity on equivalence classes of normal framings of $f$.
		
		Finally, suppose that $F$ induces the identity on equivalence classes of normal framings. Let $\varphi$ be a normal framing and let $\varphi_t$ be an extension along $F_t$. Then $\varphi=\varphi_0$ is equivalent to $\varphi_1$. By modifying $\varphi_t$ for $t\in[1-\varepsilon,1]$ for $\varepsilon>0$ small according to the isotopy between $\varphi$ and $\varphi_1$, we can assume that $\varphi=\varphi_0=\varphi_1$. Then we define the embedding $T^2\times D^{n-1}\hookrightarrow S^1\times M$,
		\[ (\lambda,e^{2\pi i t},v)\mapsto(e^{2\pi i t},\varphi_t(\lambda,v)), \]
		showing that the embedding $\overset{\circ}{F}\colon T^2\hookrightarrow S^1\times M$, $(\lambda,e^{2\pi i t})\mapsto (e^{2\pi i t},F_t(\lambda))$ has trivial normal bundle. On the other hand, the normal bundle of $\overset{\circ}{F}(T^2)$ is isomorphic to the sum of the normal bundle $\nu_F$ of $F$ in $M$ with the trivial bundle $\underline{\R}_{T^2}$ (corresponding to paths of the form $s\mapsto (e^{2\pi i (t+s)},F_t(\lambda))$). Thus, the bundle $\nu_F\oplus \underline{\R}_{T^2}$ is trivial. Since $n\geq 5$, it follows from Proposition \ref{P:BUNDLES_SURFACES}, together with the stability of the Stiefel--Whitney classes, that $\nu_F$ is trivial.
    \end{proof}

   For the existence of embeddings $F\colon T^2\hookrightarrow M$ that reverse the framing according to Lemma \ref{L:F*_W2} we have the following result.
    \begin{lemma}
        \label{L:EXISTENCE_EMBEDDING}
        Let $f\colon S^1\hookrightarrow M$ be an embedding. If there exists a continuous map $h\colon S^2\to M$ such that $w_2(M)\frown h_*[S^2]_{\Z/2}\neq 0$, then there exists an embedding $F\colon T^2\hookrightarrow M$ with $F(\cdot,1)=f$ and $w_2(M)\frown F_*[T^2]_{\Z/2}\neq 0$. If $f$ is null-homotopic, then also the converse holds.
    \end{lemma}
    \begin{proof}
        First suppose that such a map $h$ exists. By Theorem \ref{T:W-EMBEDDING}, we can assume that $h$ is an embedding. Again by Theorem \ref{T:W-EMBEDDING}, the map $T^2\to M,\, (\lambda_1,\lambda_2)\mapsto f(\lambda_1)$, which induces the trivial map on $H_2$, is homotopic to an embedding $F_0$ that extends $f$. Now, the connected sum embedding $T^2\# S^2\cong T^2\hookrightarrow M$ of $F_0$ and $h$ satisfies the required properties.

        Conversely, suppose that $f$ is null-homotopic and that $F\colon T^2\hookrightarrow M$ is an embedding with $F(\cdot,1)=f$ and $w_2(M)\frown F_*[T^2]_{\Z/2}\neq0$. Since $f$ is null-homotopic, there exists an embedding $\bar{f}\colon D^2\hookrightarrow M$ with $\bar{f}|_{S^1}=f$. We now define a map $\overline{F}\colon (D^2\times S^1)\setminus (D^3)^\circ\to M$ from the punctured solid torus to $M$ by first defining it on $T^2\cup (D^2\times \{1\})$ (and assume that the deleted $D^3$ is disjoint from this part) by setting
        \[\overline{F}(x)=F(x)\text{ for }x\in T^2\text{ and }\overline{F}(y,1)=\bar{f}(y)\text{ for }y\in D^2. \]
        Since $T^2\cup (D^2\times\{1\})$ is a deformation retract of $(D^2\times S^1)\setminus(D^3)^\circ$, we can extend this map to all of $(D^2\times S^1)\setminus(D^3)^\circ$ and we have $\overline{F}|_{T^2}=F$ by construction. We define $h\colon S^2\to M$ as the restriction of $\overline{F}$ to the other boundary component. Since $T^2$ and $S^2$ define the same homology classes inside $(D^2\times S^1)\setminus(D^3)^\circ$, it follows that they induce the same homology class $F_*[T^2]=h_*[S^2]$.
    \end{proof}
    
    Note that the assumptions of Lemma \ref{L:EXISTENCE_EMBEDDING} are satisfied, for example, if $M$ is non-spin and simply-connected.

    We will now consider the map
    \[\mu\colon \CIso{M}\to [S^1,M] \]
    given by forgetting the framing. We denote by $\widetilde{M}$ the universal cover of $M$. Since the projection $\widetilde{M}\xrightarrow{\pi} M$ is a local diffeomorphism, we have that $w_i(\widetilde{M})=\pi^*w_i(M)$. Hence, $\widetilde{M}$ is spin if and only if $w_2(M)$ lies in the kernel of $\pi^*\colon H^2(M;\Z/2)\to H^2(\widetilde{M};\Z/2)$. We now have the following proposition (cf.\ \cite{GL71} and \cite[Corollary 2.3]{Duan} in the simply-connected case).
    	
	
	\begin{proposition}
		\label{P:NUM_ISOT_CLASSES}
        The map $\mu$ is surjective. Further, we have:
		\begin{itemize}
			\item[\emph{(i)}] If $M$ is spin, then $\mu$ is two-to-one.
			\item[\emph{(ii)}] If $\widetilde{M}$ is non-spin, then $\mu$ is bijective.
            \item[\emph{(iii)}] The trivial class in $[S^1,M]$ has two preimages under $\mu$ if and only if $\widetilde{M}$ is spin. Otherwise, it has one preimage.
		\end{itemize}
	\end{proposition}
	
	
	\begin{proof}
		Since, by Theorem~\ref{T:W-EMBEDDING}, any map $S^1\to M$ is homotopic to an embedding, the map $\mu$ is surjective. Further, since any two homotopic embeddings $S^1\hookrightarrow M$ are isotopic by Theorem~\ref{T:HOM-ISO}, the preimages of a map $f\colon S^1\to M$ under $\mu$, which we can assume to be embeddings, can be represented by normal framings of $f$. This shows that there are either one or two preimages of the class represented by $f$, depending on whether the two non-equivalent normal framings of $f$ are isotopic or not.
		
		By Theorem \ref{T:W-EMBEDDING} we can assume that every self-isotopy of $f$ is an embedded torus. Hence, by Lemma \ref{L:F*_W2}, the two non-equivalent normal framings of $f$ are isotopic if and only if there is an embedded torus $F\colon T^2\to M$ with $F(\cdot,1)=f$ and $w_2(M)\frown F_*[T^2]_{\Z/2}\neq0$. In particular, if $M$ is spin, then there is no such isotopy, which shows item \emph{(i)}. Further, by Lemma \ref{L:EXISTENCE_EMBEDDING}, a sufficient condition for the existence of such an embedding is the existence of a continuous map $h\colon S^2\to M$ with $w_2(M)\frown h_*[S^2]_{\Z/2}\neq 0$, and this condition is also necessary if $f$ is null-homotopic. We now show that this condition is satisfied on $M$ if and only if it is satisfied on $\widetilde{M}$, showing items \emph{(ii)} and \emph{(iii)}.

        For any map $h\colon S^2\to \widetilde{M}$ we have
        \[w_2(M)\frown (\pi\circ h)_*[S^2]=w_2(M)\frown \pi_*h_*[S^2]=\pi_*(\pi^*w_2(M)\frown h_*[S^2])=\pi_*(w_2(\widetilde{M})\frown h_*[S^2]). \]
        Since $\pi$ induces an isomorphism on $H_0$, it follows that $h$ satisfies the required property for $\widetilde{M}$ if and only if $\pi\circ h$ satisfies it for $M$. Since any map $S^2\to M$ can be lifted to $\widetilde{M}$, the claim follows.
	\end{proof}

    By Proposition \ref{P:NUM_ISOT_CLASSES}, a null-homotopic embedding $S^1\hookrightarrow M$ can possibly have two non-isotopic normal framings. In this case we have a distinguished normal framing, which we now define.
    
    \begin{definition}\label{D:trivial_framing}
		Let $f\colon S^1\hookrightarrow M$ be an embedding. A normal framing $\varphi\colon S^1\times D^{n-1}\hookrightarrow M$ of $f$ is \emph{trivial}, if there is an embedding $\bar{\varphi}\colon D^2\times D^{n-2}$ with
		\[ \bar{\varphi}|_{S^1\times D^{n-2}}=\varphi|_{S^1\times D^{n-2}}. \]
		We say that $\varphi$ \emph{extends} over the embedded 2-disc $\bar{\varphi}(\cdot,0)$.
	\end{definition}
    It is a direct consequence of the Disc Theorem of Palais (Theorem \ref{T:DISK-ISOTOPY}) and the uniqueness of tubular neighborhoods, that any two trivial normal framings are isotopic (cf.\ also \cite[Lemma 3.7]{Re22}). Hence, there is precisely one isotopy class of trivial normal framings.
    
    \begin{lemma}
        \label{L:TRIV_FRAMING}
        Let $\varphi$ be a normal framing of a null-homotopic embedding $f\colon S^1\hookrightarrow M$. Then $\varphi$ is trivial if and only if its lift to $\widetilde{M}$ is trivial.
    \end{lemma}
    \begin{proof}
        Since we can lift extensions to the universal cover, it follows that any lift of a trivial normal framing is trivial. Since, by Proposition \ref{P:NUM_ISOT_CLASSES}, the numbers of isotopy classes of normal framings of null-homotopic embeddings for $M$ and $\widetilde{M}$ coincide, it follows that a normal framing is trivial if and only if its lift to the universal cover is trivial.
    \end{proof}
    
     \begin{definition}
    \label{DEF:STANDARD_FRAMING}
		Let $P\xrightarrow{\pi}M$ be a principal $S^1$-bundle and let $U\times S^1\cong \pi^{-1}(U)\subseteq P$ with $U\subseteq M$ open be a local trivialization. Let $D^n\hookrightarrow U$ be an orientation-preserving embedding. The corresponding embedding $S^1\times D^n\cong D^n\times S^1\hookrightarrow P$, denoted $\varphi_\pi$, is called the \emph{standard framing} of $\pi$.
	\end{definition}
	By Theorem \ref{T:DISK-ISOTOPY}, and since $\pi$ has connected structure group, the definition of standard framing is well-defined up to isotopy.
    \begin{proposition}
		\label{P:FRAMING_TRIV}
		Let $P\xrightarrow{\pi} M^n$ be a principal $S^1$-bundle such that the inclusion of a fiber is null-homotopic (which, by Lemma \ref{L:PR_BDL_TOPOLOGY2}, is equivalent to the pull-back of the Euler class to $\widetilde{M}$ being primitive). Then the standard framing $\varphi_\pi$ is trivial if and only if $\widetilde{M}$ is not spin.
	\end{proposition}
	\begin{proof}
        By Lemma \ref{L:PR_BDL_TOPOLOGY2}, since the inclusion of a fiber is null-homotopic, the pull-back of the Euler class of $\pi$ to $\widetilde{M}$ is primitive. This implies that the pull-back of $P$ along the projection $\widetilde{M}\to M$ is simply-connected by Lemma \ref{L:PR_BDL_TOPOLOGY}, in particular it is the universal cover $\widetilde{P}$, so we can write $\widetilde{P}\xrightarrow{\widetilde{\pi}}\widetilde{M}$ for the pull-back bundle. Hence, by Lemma \ref{L:TRIV_FRAMING}, $\varphi_\pi$ is trivial if and only if $\varphi_{\widetilde{\pi}}$ is trivial, that is, we can assume that $M$ and $P$ are simply-connected. This case now follows from \cite[Theorem 8]{GL71}.
    \end{proof}

 
	\section{Surgery along framed circles}
    \label{S:SURGERY}
     
     As in the previous section, $M$ denotes an oriented manifold of dimension $n\geq 5$. In this section, we consider the manifold we obtain when performing surgery along a fixed normal framing to establish the existence of certain self-diffeomorphisms of $M\#(S^2\times S^{n-2})$ and $M\#(S^2\ttimes S^{n-2})$. The technique we use is due to Wall \cite{Wa64}, who considered the corresponding problem in dimension $4$. As customary, we will assume that all corners have been smoothed after performing surgery.
	
	\begin{lemma}
		\label{L:NULL-HTPIC}
		Let $\varphi$ be a normal framing of an embedding $f\colon S^1\hookrightarrow M$ into a manifold $M$ and suppose that $f$ is null-homotopic, i.e.\ it bounds an embedded disc. Then the manifold obtained from $M$ by surgery along $\varphi$ is diffeomorphic to the connected sum of $M$ with a linear sphere bundle over $S^2$ which is trivial if and only if $\varphi$ is trivial.
	\end{lemma}
	\begin{proof}
	    Since $f$ is null-homotopic, it bounds an embedded disc by Theorem \ref{T:W-EMBEDDING}. Then the statement of the lemma is well-known. For completeness, we give the proof below.
     
        We can write $M$ as
        \begin{equation*}
            M\cong M\# S^n\cong M\# (D^2\times S^{n-2}\cup_{\mathrm{id}_{S^1\times S^{n-2}}}S^1\times D^{n-1})
        \end{equation*}
        and the inclusion $\varphi_0$ of $S^1\times D^{n-1}$ into the second factor is a trivial normal framing, the extension $\bar{\varphi}_0\colon D^2\times D^{n-2}\hookrightarrow (D^2\times S^{n-2})\cup_{S^1\times S^{n-2}} (S^1\times D^{n-1})$ is given by
        \[D^2\times D^{n-2}\cong(D^2\times D^{n-2}) \cup_{S^1\times D^{n-2}}(S^1\times D^{n-1}_+)\hookrightarrow (D^2\times S^{n-2})\cup_{S^1\times S^{n-2}} (S^1\times D^{n-1}), \]
        where $D^{n-1}_+\subseteq D^{n-1}$ denotes the upper half-ball. We use the obvious embedding on each factor and embed $D^{n-2}\subseteq S^{n-2}$ as the upper half-sphere.

        Hence, if $\varphi$ is trivial, then it is isotopic to $\varphi_0$ (as noted after Definition \ref{D:trivial_framing}), so
        the manifold obtained by surgery along $\varphi$ is diffeomorphic to
        \begin{equation}
            \label{EQ:M+S2xSn-2}
            M\# (D^2\times S^{n-2}\cup_{\mathrm{id}_{S^1\times S^{n-2}}} D^2\times S^{n-2})\cong M\# (S^2\times S^{n-2}).
        \end{equation}
        If $\varphi$ is non-trivial, then $\varphi\circ\tilde{\alpha}$ is trivial, where $\tilde{\alpha}\colon S^1\times D^{n-1}\to S^1\times D^{n-1}$ is defined by $\tilde{\alpha}(\lambda,v)=(\lambda,\alpha_\lambda v)$ and $\alpha$ is a smooth representative of the unique non-trivial class in $\pi_1(\mathrm{SO}(n-1))$. Hence, $\varphi$ is isotopic to $\varphi_0\circ\tilde{\alpha}$, so the manifold obtained by surgery along $\varphi$ is diffeomorphic to 
        \begin{equation}
            \label{EQ:M+S2~xSn-2}
            M\# (D^2\times S^{n-2}\cup_{\tilde{\alpha}} D^2\times S^{n-2})\cong M\# (S^2\ttimes S^{n-2}).
        \end{equation}
	\end{proof}
	
    The following result was already proven by Goldstein and Lininger in \cite{GL71} in the simply-connected case.
	
	
	\begin{corollary}
	    \label{C:CONN_SUM_DIFF}
		Let $M$ be a closed oriented manifold of dimension $n\geq 5$ and suppose that $w_2(\widetilde{M})\neq0$. Then $M\# (S^2\times S^{n-2})$ is diffeomorphic to $M\# (S^2\ttimes  S^{n-2})$.
	\end{corollary}
	
	
	\begin{proof}
	    By Proposition \ref{P:NUM_ISOT_CLASSES}, since $w_2(\widetilde{M})\neq 0$, the two non-equivalent normal framings of an embedded null-homotopic circle $f\colon S^1\hookrightarrow M$ are isotopic, so surgery along these framings results in diffeomorphic manifolds.
	    
	    Now fix an embedded 2-disc bounded by $f$ together with a normal framing that extends over this 2-disc. Then a normal framing representing the other equivalence class does not extend over this disc. By Lemma \ref{L:NULL-HTPIC}, if we perform surgery along these normal framings, we therefore obtain the manifolds $M\#(S^2\times S^{n-2})$ and $M\#(S^2\ttimes  S^{n-2})$, respectively.
	\end{proof}

    Now fix a null-homotopic embedding $f\colon S^1\hookrightarrow M$, a normal framing $\varphi\colon S^1\times D^{n-1}\hookrightarrow M$ of $f$, and an embedding $F\colon T^2\to M$ with $F(\cdot,1)=f$. We extend $\varphi$ along $F$, i.e.\ we obtain an isotopy $\varphi_t$ of $\varphi$ with $\varphi_t(\cdot,0)=F(\cdot,e^{2\pi i t})$. We can assume that $\varphi_1=\varphi$ or $\varphi_1=\varphi\circ\tilde{\alpha}$, depending on whether $\varphi$ and $\varphi_1$ are equivalent or not. By the isotopy extension theorem (see e.g.\  \cite[Theorem 8.1.4]{Hi76}) we can extend $\varphi_t$ to a diffeotopy $\Phi_t$ of $M$. In particular, $\Phi_0=\textup{id}_M$ and $\Phi_1$ is a diffeomorphism of $M$ which fixes $f(S^1)$ pointwise.
	
    We denote by $M_t$ the manifold obtained from $M$ by surgery along the embedding $\varphi_t$. Then all the manifolds $M_t$ are diffeomorphic, with a diffeomorphism between $M_0$ and $M_t$ induced by $\Phi_t$. It follows from Lemma~\ref{L:NULL-HTPIC} that $M_0$ is diffeomorphic to the connected sum of $M$ and a linear sphere bundle over $S^2$ and if we choose the normal framing to be trivial, then $M_0\cong M\# (S^2\times S^{n-2})$. Hence, $M_1\cong M\# (S^2\times S^{n-2})$ and if $\varphi_1$ and $\varphi_0$ are non-equivalent, we also have $M_1\cong M\# (S^2\ttimes  S^{n-2})$.
	
	We now consider the map induced by $\Phi_1$ on (co)homology.  We denote the free part of $H^i(M)$ by $H^i_F(M)$, which is the quotient of $H^i(M)$ by its torsion subgroup. Let $x_i\in H_2(M_i)$ correspond to a generator of $H_2(S^2\times S^{n-2})$ or $H_2(S^2\ttimes S^{n-2})$ (depending on whether $\varphi_0$ and $\varphi_1$ are equivalent or not) and let  $x_i^*\in H^2(M_i)$ be its dual. We then have
    \[ H_2(M_i)\cong H_2(M)\oplus\Z x_i\quad\text{and}\quad H^2(M_i)\cong H^2(M)\oplus\Z x_i^*. \]

    Note that $x_i$ is represented by the inclusion of the first factor for $S^2\times S^{n-2}$ and by a section of the base for $S^2\ttimes S^{n-2}$.

	 
	\begin{proposition}
	    \label{P:IND_HOM}
    	For the induced map ${\Phi_1}_*\colon H_2(M_0)\to H_2(M_1)$, we have ${\Phi_1}_*(y)=y$ for any $y\in H_2(M)$ and ${\Phi_1}_*(x_0)=x_1+\xi$, where $\xi=F_*[T^2]$. Analogously, the induced map $\Phi_1^*\colon H^2_F(M_i)\to H^2_F(M_i)$ on the free part is given by $\Phi_1^*\varphi=\varphi+(\varphi\frown \xi)x_i^*$.
	\end{proposition}

    For the proof of Proposition \ref{P:IND_HOM} we need the following result.

    
    \begin{lemma}
        \label{L:ISOTOPY}
        Let $M$ be a manifold and let $\iota\colon W\hookrightarrow M$ be an embedding of a manifold $W$ with non-empty boundary $N$. Let $\phi\colon[0,1]\times M\to M$ be a diffeotopy of $M$ such that $\phi_0=\mathrm{id}_M$. We define the map $\iota_\phi\colon W\to M$ as follows: Fix a diffeomorphism $W\cong W\cup_N([0,1]\times N)$ and set $\iota_\phi|_W=\iota|_W$ and $\iota_\phi(t,p)=\phi_t(\iota(p))$ for $(t,p)\in[0,1]\times N$. Then $\iota_\phi$ and $\phi_1\circ\iota$ are homotopic rel $N$.
    \end{lemma}
    The preceding lemma asserts, in short, that the homotopy class of $\phi_1\circ\iota$ rel $N$ only differs from that of $\iota$ in a collar neighborhood of $N$, where we modify it by $\Phi$.
    \begin{proof}
        We give the homotopy explicitly as follows. Define
        \[\Psi\colon[0,1]\times (W\cup_N [0,1]\times N)\to M\]
        by $\Psi_t(p)=\Phi_t(\iota(p))$ for $p\in W$ and $\Psi_t(s,p)=\Phi_{(1-s)t+s}(\iota(p))$ for $(s,p)\in[0,1]\times N$. Then $\Psi_0=\iota_\phi$, and $\Psi_1$ equals $\Phi_1\circ\iota$ on $W$ and $\Phi_1\circ\iota\circ\mathrm{pr}_N$ on $[0,1]\times N$, which, under the identification $W\cong W\cup_N([0,1]\times N)$ is homotopic rel $N$ to $\Phi_1\circ\iota$. Further, $\Psi_t|_N=\Phi_1\circ\iota|_N$ for all $t\in[0,1]$, showing that $\Psi$ is a homotopy rel $N$.
    \end{proof}
	
	\begin{proof}[Proof of Proposition \ref{P:IND_HOM}]
    	The long exact sequence in homology for the pair $(M,M\setminus \varphi(S^1\times D^{n-1}))$ yields the exact sequence
    	\begin{align}
    	    \label{EQ:LES.1}
    	    	H_3(M,M\setminus \varphi(S^1\times D^{n-1}))\to H_2(M\setminus \varphi(S^1\times D^{n-1}))\to H_2(M)\to H_2(M,M\setminus \varphi(S^1\times D^{n-1})). 
    	\end{align}
    	By excision, 
    	\[H_i(M,M\setminus \varphi(S^1\times D^{n-1}))\cong H_i(S^1\times D^{n-1},S^1\times S^{n-2})=0\]
    	for $i=2,3$, as $n\geq 5$. Hence, the map $H_2(M\setminus \varphi(S^1\times D^{n-1}))\to H_2(M)$ is an isomorphism, i.e.\
        \begin{align}
        \label{EQ:H2.M}
        H_2(M\setminus \varphi(S^1\times D^{n-1}))\cong H_2(M).
        \end{align}
    	
    	Now, consider the manifold $M_i$ for $i=0,1$. The long exact sequence in homology for the pair  $(M_i,M_i\setminus (D^2\times S^{n-2}))$ yields the exact sequence
    		\begin{align}
    	    \label{EQ:LES.2}
    	    	&H_3(M_i,M_i\setminus (D^2\times S^{n-2})) \longrightarrow  H_2(M_i\setminus (D^2\times S^{n-2})) \longrightarrow H_2(M_i)\\\longrightarrow & H_2(M_i,M_i\setminus
    	    	(D^2\times S^{n-2}))\longrightarrow H_1(M_i\setminus (D^2\times S^{n-2}))\longrightarrow H_1(M_i). \nonumber
    	\end{align}

    	As for the pair $(M,M\setminus f(S^1\times D^{n-1}))$, by excision, 
    	\[H_j(M_i,M_i\setminus (D^2\times S^{n-2}))\cong H_j(D^2\times S^{n-2},S^1\times S^{n-2})\]
    	and $H_j(D^2\times S^{n-2},S^1\times S^{n-2})$ vanishes for $j=3$ and is isomorphic to $\Z$ if $j=2$. Further, since the inclusion $M_i\setminus (D^2\times S^{n-2})\hookrightarrow M_i$ induces an isomorphism on fundamental groups,  we may rewrite the exact sequence~\eqref{EQ:LES.2} as	
    	\begin{align}
        \label{EQ:LES.3}
     0\longrightarrow H_2(M_i\setminus (D^2\times S^{n-2}))\longrightarrow H_2(M_i)\longrightarrow H_2(D^2\times S^{n-2},S^1\times S^{n-2})\longrightarrow0. 
        \end{align}
        By construction (cf.\ \eqref{EQ:M+S2xSn-2} and \eqref{EQ:M+S2~xSn-2}), the element $x_i\in H_2(M_i)$ maps to a generator of $H_2(D^2\times S^{n-2},S^1\times S^{n-2})$ in \eqref{EQ:LES.3}, since a generator of the latter is represented by $D^2\times\{v\}$ for any $v\in S^{n-2}$. We choose $v$ (by possibly modifying the map $\alpha$) so that $\alpha_\lambda(v)=v$ for all $\lambda\in S^1$ and we denote by $S$ the embedded $2$-disc in the first $(D^2\times S^{n-2})$-factor in \eqref{EQ:M+S2xSn-2} or \eqref{EQ:M+S2~xSn-2} given by $D^2\times\{v\}$. Hence, when glued to $D^2\times\{v\}$ in the second $(D^2\times S^{n-2})$-factor in \eqref{EQ:M+S2xSn-2} or \eqref{EQ:M+S2~xSn-2}, the disc $S$ represents the class $x_i$. Note that in $M\setminus\varphi(S^1\times D^{n-1})$, the surface $S$ has boundary $\varphi(S^1\times\{v\})$.

        Since $\Phi_0=\textup{id}_M$, the diffeomorphism $\Phi_1$ is isotopic to a map that fixes $M\setminus \varphi(S^1\times D^{n-1})$ pointwise. Hence, $\Phi_1$ induces the identity in homology for all classes in $H_*(M)$. Now, by Lemma \ref{L:ISOTOPY}, the inclusion of $S$, which we will denote by $\iota_S$, followed by $\Phi_1$ is homotopic rel $\varphi(S^1\times\{v\})$ to $\iota_S$ extended by the map $\tilde{F}\colon [0,1]\times S^1\to M$ defined by $\tilde{F}(t,\lambda)=\Phi_t(\varphi(\lambda,v))$, where we identify $S$ with $S\cup_{S^1}([0,1]\times S^1)$. It follows that the surface representing $x_0$ is mapped to a surface which represents the class $x_1+\xi$. This can be seen in a similar way as in the proof of Lemma \ref{L:EXISTENCE_EMBEDDING}: The map $\tilde{F}$ and the inclusions of $S$ and $(D^2\times \{v\})$ all coincide on their boundaries, hence they define a map from $T^2$ with $2$ discs glued into $S^1\times\{1\}$, and therefore define a map from the twice punctured solid torus into $M_1$. The map restricted to each boundary component represents $\xi$, $x_1$ and $\Phi_* x_0$, respectively. Hence, after a suitable choice of orientations, we obtain $\Phi_* x_0=x_1+\xi$.

        Finally, since $H_1(M_i)\cong H_1(M)$, the statement on the cohomology follows from the universal coefficient theorem.
    \end{proof}

	In the following, $x$ and $\tilde{x}$ denote, respectively, generators of $H_2(S^2\times S^{n-2})$ and $H_2(S^2\ttimes  S^{n-2})$, and $x^*$ and $\tilde{x}^*$ denote the corresponding dual elements in $H^2(S^2\times S^{n-2})$ and $H^2(S^2\ttimes  S^{n-2})$, i.e.\ $x^* \frown x=1$ and $\tilde{x}^*\frown \tilde{x} = 1$. The following corollaries now directly follow from Lemma \ref{L:EXISTENCE_EMBEDDING} and Proposition \ref{P:IND_HOM}.
	\begin{corollary}
		\label{C:HOMOLOGY_DIFF_CS}
		Let $M$ be a closed, oriented manifold of dimension $n\geq 5$. Then, for any continuous map $h\colon S^2\to M$ with $w_2(M)\frown h_*[S^2]_{\Z/2}= 0$, there is a diffeomorphism of $M\# (S^2\times S^{n-2})$ which induces the identity on $H_2(M)$ and maps $x$ to $x+\xi$, where $\xi=h_*[S^2]$. The induced map on cohomology fixes $x^*$ and maps $\varphi\in H^2_F(M)$ to $\varphi+\varphi(\xi)x^*$. An analogous statement holds if we replace $S^2\times S^{n-2}$ with $S^2\ttimes  S^{n-2}$, $x$ with $\tilde{x}$, and $x^*$ with $\tilde{x}^*$.
	\end{corollary}

	\begin{corollary}
		\label{C:HOMOLOGY_DIFF_CS_NONSP}
		Let $M$ be a closed, oriented manifold of dimension $n\geq 5$. Then, for any continuous map $h\colon S^2\to M$ with $w_2(M)\frown h_*[S^2]_{\Z/2}\neq 0$, there is a diffeomorphism between $M\# (S^2\times S^{n-2})$ and $M\# (S^2\ttimes  S^{n-2})$ which induces the identity on $H_2(M)$ and induces the map $x\mapsto \tilde{x}+\xi$, where $\xi=h_*[S^2]$. The induced map on cohomology maps $x^*$ to $\tilde{x}^*$ and maps $\varphi\in H^2_F(M)$ to $\varphi+\varphi(\xi)\tilde{x}^*$.
	\end{corollary}

The following corollary is an analog of a result of Wall for $4$-manifolds (see \cite[Theorem 2]{Wa64}).

    \begin{corollary}
        \label{C:HOMOLOGYDIFF}
        Let $M_1$ and $M_2$ be $k$-fold connected sums of copies of $S^2\times S^{n-2}$ and $S^2\ttimes S^{n-2}$. Then every isomorphism between $H^2(M_1)$ and $H^2(M_2)$ that preserves $w_2$ is induced by a diffeomorphism. In particular, every isomorphism of the second cohomology of $\#_k(S^2\times S^{n-2})$ is induced by a diffeomorphism.
	\end{corollary}
	\begin{proof}
	    We first consider the case where $M_1$ and $M_2$ are both spin, i.e.\ they are both diffeomorphic to $\#_k (S^2\times S^{n-2})$, which we will denote by $N_k$. Denote a generator of $H^2(S^2\times S^{n-2})$ in the $i$-th summand of $N_k$ by $x_i$. Then $(x_1,\dots,x_k)$ is a basis of $H^2(N_k)$. The automorphism group of $H^2(N_k)$ can be identified with $\mathrm{GL}(k,\Z)$ and, by applying Corollary \ref{C:HOMOLOGY_DIFF_CS} to the $i$-th summand of $N_k$ with $\xi$ a multiple of the dual of $x_j$, $i\neq j$, we obtain that all elementary matrices are induced by a diffeomorphism. Since the elementary matrices together with the permutation matrices, which are obviously induced by diffeomorphisms, generate $\mathrm{GL}(k,\Z)$, the claim follows.
	    
	    If $M_1$ and $M_2$ are non-spin, by applying Corollary \ref{C:CONN_SUM_DIFF} (possibly multiple times), we can assume that both $M_1$ and $M_2$ are diffeomorphic to a fixed connected sum of copies of $S^2\times S^{n-2}$ and $S^2\ttimes S^{n-2}$, where the latter appears at least once, and we denote this manifold by $N_k'$. As before, we denote by $x_i$ a generator of the second cohomology of the $i$-th summand of $N_k'$, so $(x_1,\dots,x_k)$ is a basis of $H^2(M_k')$. By Corollaries \ref{C:HOMOLOGY_DIFF_CS} and \ref{C:HOMOLOGY_DIFF_CS_NONSP}, we see as in the spin case that every automorphism of $H^2(N_k')$ is induced by a diffeomorphism if we allow the bundle structure of the summands to change. By restricting to those automorphisms that fix $w_2(N_k')$ we obtain all diffeomorphisms that do not change the bundle structures of the summands, i.e.\ all self-diffeomorphisms of $N_k'$.
	\end{proof}

    
	\section{Twisted suspensions}

    \label{S:TW_SUSP}

    Let $M^n$ be a connected $n$-manifold and let $e\in H^2(M;\Z)$. Generalizing Duan's suspension constructions in \cite{Duan}, for a class $e\in H^2(M;\Z)$ we now define two $(n+1)$-dimensional manifolds $\Sigma_e M$ and $\widetilde{\Sigma}_e M$, called \emph{suspensions of $M$ twisted by $e$}, as follows.
    
    The class $e$ defines a unique principal $S^1$-bundle $P\xrightarrow{\pi}M$ with Euler class $e(\pi)=e$. Let $D^n\hookrightarrow M$ be an embedding. If $M$ is orientable, we require, after choosing an orientation on $M$, that this embedding be orientation-preserving. Since $D^n$ is contractible, we can identify $\pi^{-1}(D^n)$ with $D^n\times S^1$ and we obtain an $S^1$-equivariant embedding
    \[\varphi_\pi\colon D^n\times S^1\hookrightarrow P.\]
    The definition of $\varphi_\pi$ is unique up to isotopy. This follows from the fact that the embedding $D^n\hookrightarrow M$ is unique up to isotopy by Theorem \ref{T:DISK-ISOTOPY} and that $S^1$ is connected, so the identification of $\pi^{-1}(D^n)$ with $D^n\times S^1$ is unique.

    \begin{definition}
        Assume $n\geq 2$ and let $\alpha\colon S^1\to\mathrm{SO}(n)$ be a smooth representative of a generator of $\pi_1(\mathrm{SO}(2))\cong\Z$ if $n=2$ and of the unique non-trivial class in $\pi_1(\mathrm{SO}(n))\cong\Z/2$ if $n>2$. The map $\alpha$ induces the diffeomorphism $\tilde{\alpha}\colon S^{n-1}\times S^1\to S^{n-1}\times S^1, (x,y)\mapsto (\alpha_y x,y)$. We define the \emph{suspensions of $M$ twisted by $e$} as
        \[\Sigma_e M = P\setminus(\varphi_\pi(D^n\times S^1)^\circ)\cup_{\mathrm{id}_{S^{n-1}\times S^1}}S^{n-1}\times D^2 \]
        and
        \[\widetilde{\Sigma}_e M = P\setminus(\varphi_\pi(D^n\times S^1)^\circ)\cup_{\tilde{\alpha}}S^{n-1}\times D^2. \] 
    \end{definition}
    When $e$ is the trivial class we recover the suspension constructions in \cite{Duan}, where they are denoted by $\Sigma_0M$ and $\Sigma_1 M$, respectively.

With the definition of twisted suspensions in hand, we now prove Theorem~\ref{T:MAIN}.

    \begin{proof}[Proof of Theorem \ref{T:MAIN}] We will follow the same strategy as in \cite[Theorem B]{Duan}. We write $P_1$ as
    \[P_1\cong P_1\# S^{n+1}\cong P_1\# (D^2\times S^{n-1}\cup_{S^1\times S^{n-1}} S^1\times D^n).\] 
        Now the inclusion $\varphi$ of $S^1\times D^n$ into the second factor is a trivial normal framing as in the proof of Lemma \ref{L:NULL-HTPIC}. By Proposition \ref{P:FRAMING_TRIV}, the normal framings $\varphi_{\pi_1}$ and $\varphi$ are isotopic if and only if $\widetilde{B}_1$ is not spin, and if $\widetilde{B}_1$ is spin, then $\varphi_{\pi_1}$ is isotopic to $\varphi\circ\tilde{\alpha}$, where we extend $\tilde{\alpha}$ to $S^1\times D^n$ in the obvious way. It follows that
        \begin{align*}
            P&\cong (P_1\setminus \varphi_{\pi_1}(S^1\times D^n)^\circ)\cup_{S^1\times S^{n-1}} (P_2\setminus \varphi_{\pi_2}(S^1\times D^n)^\circ)\\
            &\cong P_1\#(D^2\times S^{n-1})\cup_{S^1\times S^{n-1}} (P_2\setminus \varphi_{\pi_2}(S^1\times D^n)^\circ),
        \end{align*}
        and we use either $\mathrm{id}_{S^1\times S^{n-1}}$ or $\tilde{\alpha}$ as gluing map, depending on whether $\varphi_\pi$ is trivial or not. In the first case, we obtain $P\cong P_1\# \Sigma_{e_2} M_2$ and in the second case, we obtain $P\cong P_1\# \widetilde{\Sigma}_{e_2} M_2$.  
    \end{proof}

The following result yields basic topological information on twisted suspensions of manifolds.

    \begin{lemma}
        \label{L:SUSP_TOP}
        Let $M$ be a connected manifold of dimension $n\geq 2$ and let $e\in H^2(M)$. Then topological invariants of the twisted suspensions are given as follows:
        \begin{enumerate}
            \item Fundamental group:
                \[\pi_1(\Sigma_e M)\cong \pi_1(\widetilde{\Sigma}_e M)\cong\begin{cases}
            \pi_1(M),\quad & n\geq 3,\\ \pi_1(M\setminus D^2),\quad & n=2.
            \end{cases} \]
            \item The inclusions of $P\setminus(\varphi_\pi(D^n\times S^1)^\circ)=\pi^{-1}(M\setminus(D^n)^\circ)$ into $\Sigma_e M$ and $\widetilde{\Sigma}_e M$ induce isomorphisms in (co)homology in all degrees $i$ with $3\leq i\leq n$ (with coefficients in any ring).
            \item If $M$ is simply-connected and $n\geq 5$, then
            \[H^2(\Sigma_e M)\cong H^2(\widetilde{\Sigma}_e M)\cong H^2(M)\] 
            and similarly for $H_2$ (with coefficients in any ring). Further, $\Sigma_e M$ is spin if and only if $w_2(M)\equiv e\mod 2$, and $\widetilde{\Sigma}_e M$ is spin if and only if $M$ is spin.
        \end{enumerate}
    \end{lemma}
    \begin{proof}
        As before, we denote by $P\xrightarrow{\pi}M$ the principal $S^1$-bundle over $M$ with Euler class $e$. The spaces $\Sigma_e M$ and $\widetilde{\Sigma}_e M$ fit into the following pushout diagram:
        \[
        \begin{tikzcd}
            S^{n-1}\times S^1\arrow{r}{\mathrm{id}_{S^{n-1}\times S^1}\, (\text{resp. }\tilde{\alpha})}\arrow{d}{\varphi_\pi|_{S^{n-1}\times S^1}} & S^{n-1}\times D^2\arrow{d} \\
            P\setminus (\varphi_\pi(D^n\times S^1)^\circ)\arrow{r} & \Sigma_e M\, (\text{resp. }\widetilde{\Sigma}_e M)
        \end{tikzcd}
        \]
        Hence, by van Kampen's theorem, both $\pi_1(\Sigma_e M)$ and $\pi_1(\widetilde{\Sigma}_e M)$ are isomorphic to the quotient of $\pi_1(P\setminus (\varphi_\pi(D^n\times S^1)^\circ))$ by the subgroup generated by the class represented by a fiber. By the long exact sequence of homotopy groups for the $S^1$-bundle $P\setminus (\varphi_\pi(D^n\times S^1)^\circ)\xrightarrow{\pi} M\setminus {D^n}^\circ$, this quotient is isomorphic to $\pi_1(M\setminus {D^n}^\circ)$, which is isomorphic to $\pi_1(M)$ if $n\geq 3$. This proves item (1) and item (2) follows from the Mayer--Vietoris sequence for the same pushout diagram.
        
        Now, assume that $M$ is simply-connected. We consider $M'=(S^2\times S^{n-2})\# M$ and denote by $P'\xrightarrow{\pi'}M'$ the principal $S^1$-bundle with Euler class $e'=x^*+e$, where $x^*$ denotes a generator of $H^2(S^2\times S^{n-2})$. By Theorem \ref{T:MAIN}, it follows that
        \[P'\cong (S^3\times S^{n-2})\# \widetilde{\Sigma}_e M.\]
        By the Gysin sequence, we have the following exact sequence:
        \[ H^0(M')\xrightarrow{\cdot\smile e' } H^2(M')\xrightarrow{\pi'^*}H^2(P')\to 0. \]
        Hence,
        \[H^2(\widetilde{\Sigma}_e M)\cong H^2(P')\cong \bigslant{H^2(M')}{\langle e'\rangle}\cong H^2(M). \]
        By Lemma \ref{L:PR_BDL_CHAR_CL}, $w_2(P')=\pi'^*w_2(M')=\pi'^*w_2(M)$, which only lies in $\langle e'\mod 2\rangle$ when $w_2(M)$ is trivial.
        
        For $\Sigma_e M$ we proceed similarly by defining $M'=(S^2\ttimes S^{n-2})\# M$. In this case, since $w_2(S^2\ttimes S^{n-2})$ is non-trivial, $w_2(M')$ lies in $\langle e'\mod 2\rangle$ if and only if $w_2(M)\equiv e\mod 2$. This proves item (3).
    \end{proof}

    
    \begin{example}
        \label{EX:SPHERE_SUSP}
        We can explicitly determine the diffeomorphism type of the twisted suspension in the following cases:
        \begin{enumerate}
            \item We have, for $n\geq 2$,
            \[ \Sigma_0 S^n = (D^n\times S^1)\cup_{\mathrm{id}_{S^{n-1}\times S^1}}(S^{n-1}\times D^2)\cong \partial (D^n\times D^2)\cong \partial D^{n+2}=S^{n+1}. \]
            This also holds for $\widetilde{\Sigma}_0 S^n$, as the diffeomorphism $\tilde{\alpha}$ extends over the right-hand side, i.e.\ over $(D^n\times S^1)$.
            \item For $n=2$ and $e\in H^2(S^2)$ non-trivial, we also obtain that
            \[\Sigma_e S^2\cong \widetilde{\Sigma}_e S^2\cong S^3, \]
            since, by Lemma \ref{L:SUSP_TOP}, both $\Sigma_e S^2$ and $\widetilde{\Sigma}_e S^2$ are closed, simply-connected $3$-manifolds, which, by Perelman's proof of the Poincaré conjecture, can only be diffeomorphic to $S^3$.
            \item If $e\in H^2(\C P^n)$ denotes a generator, then we have
            \[ \Sigma_e\C P^n\cong\begin{cases}
                S^2\times S^{2n-1},\quad & n\text{ even},\\
                S^2\ttimes S^{2n-1},\quad & n\text{ odd},
            \end{cases}\quad\quad
            \widetilde{\Sigma}_e\C P^n\cong\begin{cases}
                S^2\ttimes S^{2n-1},\quad & n\text{ even},\\
                S^2\times S^{2n-1},\quad & n\text{ odd}.
            \end{cases}
            \]
            This will follow immediately from part (1) of Theorem \ref{T:SUSP_EX}.
        \end{enumerate}
    \end{example}

    Now, let $E\xrightarrow{\xi} M^n$ be a fiber bundle with fiber $F$. For $e\in H^2(M)$ we construct a fiber bundle $\Sigma_e\xi$ (resp.\ $\widetilde{\Sigma}_e\xi$) over $\Sigma_e M$ (resp.\ $\widetilde{\Sigma}_e M$) with fiber $F$ and the same structure group as $\xi$ as follows. Let $D^n\subseteq M$ be an embedded disc and extend it to local trivializations $\varphi_\xi\colon D^n\times F\hookrightarrow E$ and $\varphi_\pi\colon D^n\times S^1\hookrightarrow P$, where $P\xrightarrow{\pi} M$ denotes, as before, the principal $S^1$-bundle over $M$ with Euler class $e$. The pull-back $\pi^*(E\setminus \varphi_\xi(D^n\times F)^\circ)$ is then a fiber bundle over $P\setminus\varphi_\pi(D^n\times S^1)^\circ$ with fiber $F$, the same structure group as $\xi$, and boundary $S^{n-1}\times S^1\times F$.
    \begin{definition}
        We define $E(\Sigma_e\xi)$ and $E(\widetilde{\Sigma}_e\xi)$ by
        \[
            E(\Sigma_e\xi)=\pi^*(E\setminus \varphi_\xi(D^n\times F)^\circ)\cup_{\textup{id}_{S^{n-1}\times S^1\times F}}(S^{n-1}\times D^2\times F)
        \]
        and
        \[
            E(\widetilde{\Sigma}_e\xi)=\pi^*(E\setminus \varphi_\xi(D^n\times F)^\circ)\cup_{\tilde{\alpha}\times \textup{id}_F}(S^{n-1}\times D^2\times F).
        \]
        Since we glue fibers to fibers, where we consider the right-hand side as the trivial bundle $S^{n-1}\times D^2\times F\to S^{n-1}\times D^2$, we obtain the structure of two fiber bundles with fiber $F$, the same structure group as $\xi$, and base
        \[P\setminus\varphi_\pi(D^n\times S^1)^\circ\cup_{\textup{id}_{S^{n-1}\times S^1}}S^{n-1}\times D^2=\Sigma_e M \]
        and
        \[P\setminus\varphi_\pi(D^n\times S^1)^\circ\cup_{\tilde{\alpha}}S^{n-1}\times D^2=\widetilde{\Sigma}_e B, \]
        respectively. We denote the projection maps $E(\Sigma_e\xi)\to \Sigma_e M$ and $E(\widetilde{\Sigma}_e\xi)\to\widetilde{\Sigma}_e M$ by $\Sigma_e\xi$ and $\widetilde{\Sigma}_e\xi$, respectively.
    \end{definition}

    Now we restrict to the case of linear sphere bundles, i.e.\ let $E\xrightarrow{\xi}M^n$ be a linear $S^m$-bundle and let $e\in H^2(M)$. It follows from the corresponding  constructions that the bundle $\Sigma_e\xi$ is trivial over the right-hand side of the decomposition
    \[\Sigma_e M=P\setminus\varphi_\pi(D^n\times S^1)^\circ\cup_{\textup{id}_{S^{n-1}\times S^1}}S^{n-1}\times D^2, \]
    i.e.\ it is given by $S^{n-1}\times D^2\times S^m$ (and the construction provides a canonical identification) and similarly for the bundle $\widetilde{\Sigma}_e\xi$. By decomposing $S^m=D^m\cup_{S^{m-1}}D^m$ and identifying $D^2\times D^m\cong D^{m+2}$, we obtain embeddings
    \[ \iota_{\xi}\colon S^{n-1}\times D^{m+2}\hookrightarrow E(\Sigma_e\xi) \]
    and
    \[ \widetilde{\iota}_{\xi}\colon S^{n-1}\times D^{m+2}\hookrightarrow E(\widetilde{\Sigma}_e\xi). \]
    \begin{proposition}
        \label{P:ESIGMA_SURGERY}
        If $n\geq 2$, then the manifold $\Sigma_{\xi^*e}E$ (resp.\ $\widetilde{\Sigma}_{\xi^*e}E$) is diffeomorphic to the manifold obtained by surgery on $E(\Sigma_e\xi)$ (resp.\ $E(\widetilde{\Sigma}_e\xi)$) along the embedding $\iota_\xi$ (resp.\ $\widetilde{\iota}_\xi$).
    \end{proposition}
    \begin{proof}
        Recall that we have a local trivialization
        \[\varphi_\xi\colon D^n\times S^m\hookrightarrow E. \]
        Thus, after smoothing corners, the restriction of $\varphi_\xi$ to $D^n\times S^m_+\cong D^n\times D^m$, where $S^m_+$ denotes the (closed) upper hemisphere of $S^m$, is an orientation-preserving embedding of $D^{m+n}$ into $E$.

        It follows that in the decomposition
        \[\pi^*(E)\cong \pi^*(E\setminus\varphi_\xi(D^n\times S^m))\cup_{\textup{id}_{S^{n-1}\times S^m\times S^1}}(D^n\times S^m\times S^1) \]
        a local trivialization for $\pi^*(E)$ is given by the inclusion of $D^n\times S^m_+\times S^1$ into the right-hand side. Hence, to construct the space $\Sigma_{\xi^*c}E$ (resp.\ $\widetilde{\Sigma}_{\xi^*c}E$), we need to glue the product $S^{n+m-1}\times D^2$ to $\pi^*(E\setminus\varphi_\xi(D^n\times S^m))\cup_{\textup{id}_{S^{n-1}\times S^m_-\times S^1}}(D^n\times S^m_-\times S^1)$ along the boundary $S^{n+m-1}\times S^1$, which, in this decomposition, is given by
        \[ (S^{n-1}\times S^m_+\times S^1)\cup_{\textup{id}_{S^{n-1}\times S^{m-1}\times S^1}}(D^{n}\times S^{m-1}\times S^1). \]
        If we now decompose
        \[ (S^{n+m-1}\times D^2)\cong (S^{n-1}\times S^m_+\times D^2)\cup_{\textup{id}_{S^{n-1}\times S^{m-1}\times D^2}}(D^{n}\times S^{m-1}\times D^2), \]
        we obtain that the space $\Sigma_{\xi^*c}E$ (resp.\ $\widetilde{\Sigma}_{\xi^*c}E$) is the result of gluing according to the following diagram, where the map $\phi$ will be constructed below:
        \begin{equation}
            \label{DIAG:GLUING}
            \begin{tikzcd}[cells={nodes={draw=black, anchor=center,minimum height=2em}}]
	 	\pi^*(E\setminus\varphi_\xi(D^n\times S^m))\arrow{r}{\mathrm{id}_{S^{n-1}\times S^{m}_-\times S^1}}\arrow{d}{\phi|_{S^{n-1}\times S^m_+\times S^1} } &[5em] D^n\times S^m_-\times S^1\arrow{d}{\phi|_{D^n\times S^{m-1}\times S^1}}\\[5ex]
	 	S^{n-1}\times D^m\times D^2\arrow{r}{\mathrm{id}_{S^{n-1}\times S^{m-1}\times D^2}} &[5em] D^n\times S^{m-1}\times D^2
	 	\end{tikzcd}  
        \end{equation}
        Here, an arrow denotes gluing of the two spaces it connects along parts of their boundary via the map indicated.

        The map $\phi$ in diagram \eqref{DIAG:GLUING} is a self-diffeomorphism of
        \[(S^{n-1}\times S^m_+\times S^1)\cup_{\mathrm{id}_{S^{n-1}\times S^{m-1}\times S^1}}(D^n\times S^{m-1}\times S^1)\cong S^{n+m-1}\times S^1 \]
        defined as follows: For $\Sigma_{\xi^*c}E$, set $\phi=\mathrm{id}_{S^{n+m-1}\times S^1}$. For $\widetilde{\Sigma}_{\xi^*c}E$, let $\alpha$ be a smooth representative of a generator of $\pi_1(\mathrm{SO}(n))$ (which is isomorphic to $\Z/2$ if $n>2$ and to $\Z$ if $n=2$) and set
        \[\phi(x,y,\lambda)=(T_\lambda x,y,\lambda). \]
        We claim that $\phi$ is the gluing map in the construction of $\Sigma_{\xi^*c}E$ (resp.\ $\widetilde{\Sigma}_{\xi^*c}E$). For $\Sigma_{\xi^*c}E$, this is clear by construction. For $\widetilde{\Sigma}_{\xi^*c}E$, note that in the decomposition
        \[S^{n+m-1}\cong (S^{n-1}\times S^m_+)\cup_{\mathrm{id}_{S^{n-1}\times S^{m-1}}}(D^n\times S^{m-1}) \]
        the first factor corresponds to the embedding of a tubular neighborhood of $S^{n-1}\subseteq \R^n\subseteq\R^{n+m}$ into $S^{n+m-1}\subseteq \R^{n+m-1}$. Since the inclusion $\mathrm{SO}(n)\subseteq\mathrm{SO}(n+m)$ induces a surjection on fundamental groups (and in fact an isomorphism if $n>2$), it follows that the map $\phi$ represents the non-trivial class in $\pi_1(\mathrm{SO}(n+m))$.

        We now modify diagram \eqref{DIAG:GLUING} by noting that the map $\phi|_{S^n\times S^{m-1}\times S^1}$ extends over $D^n\times S^m_-\times S^1$ as the identity on the second factor, and we denote the extension again by $\phi$. Hence, we obtain the following gluing diagram:

        \begin{equation}
            \label{DIAG:GLUING2}
            \begin{tikzcd}[cells={nodes={draw=black, anchor=center,minimum height=2em}}]
	 	\pi^*(E\setminus\varphi_\xi(D^n\times S^m))\arrow{r}{\phi|_{S^{n-1}\times S^{m}_-\times S^1}}\arrow{d}{\phi|_{S^{n-1}\times S^m_+\times S^1} } &[5em] D^n\times S^m_-\times S^1\arrow{d}{\mathrm{id}_{D^n\times S^{m-1}\times S^1}}\\[5ex]
	 	S^{n-1}\times D^m\times D^2\arrow{r}{\mathrm{id}_{S^{n-1}\times S^{m-1}\times D^2}} &[5em] D^n\times S^{m-1}\times D^2
	 	\end{tikzcd}  
        \end{equation}

        We observe now that gluing according to the right vertical part of diagram \eqref{DIAG:GLUING2} yields the space
        \[(D^n\times S^m_-\times S^1)\cup_{\mathrm{id}_{D^n\times S^{m-1}\times S^1}}(D^n\times S^{m-1}\times D^2)\cong (D^n\times S^{m+1}), \]
        while gluing according to the left vertical part yields the space
        \begin{align*}
            \pi^*(E\setminus\varphi_\xi(D^n\times S^m))\cup_{\phi|_{S^{n-1}\times S^m_+\times S^1}}(S^{n-1}\times D^m\times D^2),
        \end{align*}
        which can be alternatively written as
        \begin{align*}
            \pi^*(E\setminus\varphi_\xi(D^n\times S^m))\cup_{\phi|_{S^{n-1}\times S^m_+\times S^1}}((S^{n-1}\times S^m\times D^2)\setminus (S^{n-1}\times S^m_-\times D^2)).
        \end{align*}
        This space is, by construction, the space $E(\Sigma_e\xi)$ (resp.\ $E(\widetilde{\Sigma}_e\xi)$) with the image of the embedding $\iota_\xi$ (resp.\ $\widetilde{\iota}_\xi$) removed. It follows that $\Sigma_{\xi^*e}E$ (resp.\ $\widetilde{\Sigma}_{\xi^*e}E$) is obtained from $E(\Sigma_e\xi)$ (resp.\ $E(\widetilde{\Sigma}_e\xi)$) by surgery along $\iota_\xi$ (resp.\ $\widetilde{\iota}_\xi$).
    \end{proof}
    
    \begin{proposition}
        \label{P:SUSP_SPH_BDL}
        Let $E\xrightarrow{\xi}S^n$ be a linear $S^m$-bundle with $m,n\geq2$. Let $T\colon S^{n-1}\to \mathrm{SO}(m+1)$ be the clutching function of $\xi$, and assume that the image of $T$ is contained in $\mathrm{SO}(m)\subseteq\mathrm{SO}(m+1)$.
        \begin{enumerate}
            \item If $n>2$, then the manifold $\Sigma_0 E$ (resp.\ $\widetilde{\Sigma}_0 E$) is diffeomorphic to the connected sum of $E(\Sigma_0\xi)$ (resp.\ $E(\widetilde{\Sigma}_0\xi)$) and the linear $S^{m+1}$-bundle over $S^n$ with clutching function given by the composition of $T$ with the inclusion $\mathrm{SO}(m+1)\subseteq\mathrm{SO}(m+2)$. In particular, if $\xi$ is trivial, i.e.\ $E=S^n\times S^m$, then both $\Sigma_0 E$ and $\widetilde{\Sigma}_0 E$ are diffeomorphic to
            \[ (S^{n+1}\times S^m)\#(S^n\times S^{m+1}). \]
            \item If $n=2$, where we have $E\cong S^2\times S^m$ or $E\cong S^2\ttimes S^m$, let $m\geq 3$ and $e\in H^2(S^2)$. We denote by $d$ the divisibility of $e$. Then
            \[ \Sigma_{\xi^*e} E\cong \begin{cases}
                (S^2\times S^{m+1})\# (S^3\times S^m),\quad & E\cong S^2\times S^m\text{ and }d\text{ is even, or }\\&E\cong S^2\ttimes S^m\text{ and }d\text{ is odd},\\
                (S^2\ttimes S^{m+1})\# (S^3\times S^m),\quad &\text{else,}
            \end{cases} \]
            and 
            \[ \widetilde{\Sigma}_{\xi^*e} E\cong \begin{cases}
                (S^2\times S^{m+1})\# (S^3\times S^m),\quad & E\cong S^2\times S^m,\\
                (S^2\ttimes S^{m+1})\# (S^3\times S^m),\quad &\text{else.}
            \end{cases} \]
        \end{enumerate}
    \end{proposition}
    \begin{proof}
        \emph{(1)}. By definition, we can decompose the spaces $E(\Sigma_0 \xi)$ and $E(\widetilde{\Sigma}_0\xi)$ as
            \begin{equation}
                \label{EQ:ESIGMA_DECOMP3}
                 E(\Sigma_0\xi)\cong (D^n\times S^1\times S^m)\cup_{\phi_1}(S^{n-1}\times D^2\times S^m)
            \end{equation}
            and
            \begin{equation}
                \label{EQ:ESIGMA_DECOMP4}
                E(\widetilde{\Sigma}_0\xi)\cong (D^n\times S^1\times S^m)\cup_{\phi_2}(S^{n-1}\times D^2\times S^m), 
            \end{equation}
            where the diffeomorphisms $\phi_1,\phi_2\colon S^{n-1}\times S^1\times S^m\to S^{n-1}\times S^1\times S^m$ are given by
            \[ \phi_1(x,y,z)=(x,y,T_x z) \]
            and
            \[\phi_2(x,y,z)=(\alpha_y x,y, T_x z). \]
            We further decompose
            \[S^{n-1}\times D^2\times S^m\cong (S^{n-1}\times D^2\times S^m_+)\cup_{\mathrm{id}_{S^{n-1}\times D^2\times S^{m-1}}}(S^{n-1}\times D^2\times S^m_-) \]
            and the embeddings $\iota_\xi$ and $\widetilde{\iota}_\xi$ are given by the inclusion of the second factor.

            Since the image of $T$ is contained in $\mathrm{SO}(m)$, we can assume that $T_x$ preserves $S^m_-$ and is given by a linear map on $S^m_-$ when identifying $S^m_-\cong D^m$. In particular, $T_x$ fixes the south pole $z_S\in S^m_-$. Further, we can deform the map $\alpha$ to be constant $\mathrm{id}_{\R^n}$ on $S^1_-$.
            
            By isotoping the embeddings $\iota_\xi$ and $\widetilde{\iota}_\xi$ to the left-hand side of \eqref{EQ:ESIGMA_DECOMP3} and \eqref{EQ:ESIGMA_DECOMP4}, respectively, we obtain in both cases the embedding
            \[ \iota\colon S^{n-1}\times D^1\times S^1_-\times S^m_-\hookrightarrow D^n\times S^1\times S^m,\quad (x,y_1,y_2,z)\mapsto ((x,\frac{1}{2}y_1),y_2,T_x^{-1}z), \]
            where we have identified $D^2\cong D^1\times S^1_-$ and $D^n$ as the space obtained from $S^{n-1}\times D^1=S^{n-1}\times [-1,1]$ by collapsing $S^{n-1}\times\{-1\}$ to a point.
            
            Now, define the map $T^\prime\colon S^{n-1}\to\mathrm{SO}(m+2)$,
            \[ T'_x(y,z)=T_xz \]
            for $x\in S^{n-1}$, $y\in \R^2$ and $z\in \R^m$. Then, when viewing $\iota$ as a normal framing of an embedding of $S^{n-1}$, modifying the framing by $T'$ yields a normal framing that extends over an embedded disc. It follows  as in Lemma \ref{L:NULL-HTPIC} (see e.g.\ \cite[Lemma 3.8]{Re22}) that the manifold obtained by surgery along the embedding $\iota_\xi$, which by Proposition \ref{P:ESIGMA_SURGERY} is diffeomorphic to $\Sigma_0 E$, is diffeomorphic to
            \[ E(\Sigma_0\xi)\#((D^n\times S^{m+1})\cup_{\tilde{T}'}(D^n\times S^{m+1}), \]
            where $\tilde{T}'\colon S^{n-1}\times S^{m+1}\to S^{n-1}\times S^{m+1}$ is defined by $\tilde{T}'(x,y)=(x,T'_x y)$, and similarly for $E(\widetilde{\Sigma}_0\xi)$. The right-hand side is the total space of the linear $S^{m+1}$-bundle over $S^n$ with clutching function $T'$.
            
        \emph{(2)}. By Proposition \ref{P:ESIGMA_SURGERY}, the manifold $\Sigma_{\xi^* e}E$ (resp.\ $\widetilde{\Sigma}_{\xi^* e}E$) is obtained by surgery on an embedding of $S^1\times D^{m+2}$ in $E(\Sigma_e\xi)$ (resp.\ $E(\widetilde{\Sigma}_e\xi)$). The spaces $E(\Sigma_e\xi)$ and $E(\widetilde{\Sigma}_e\xi)$ are total spaces of linear sphere bundles over $\Sigma_e S^2$ and $\widetilde{\Sigma}_e S^2$, respectively, which, by Example \ref{EX:SPHERE_SUSP}, are diffeomorphic to $S^3$. Since any linear sphere bundle over $S^3$ is trivial, both $E(\Sigma_e\xi)$ and $E(\widetilde{\Sigma}_e\xi)$ are diffeomorphic to $S^3\times S^m$.
            
            Since $S^3\times S^m$ is simply-connected, it follows from Lemma \ref{L:NULL-HTPIC} that $\Sigma_{\xi^* e}E$ (resp.\ $\widetilde{\Sigma}_{\xi^* e}E$) is diffeomorphic to either $(S^2\times S^{m+1})\# (S^3\times S^m)$ (which is spin) or $(S^2\ttimes S^{m+1})\# (S^3\times S^m)$ (which is non-spin). By Lemma \ref{L:SUSP_TOP}, we can characterize when $\Sigma_{\xi^* e}E$ (resp.\ $\widetilde{\Sigma}_{\xi^* e}E$) in terms of the Euler class and Stiefel--Whitney class of $E$, which yields the different cases as claimed.
    \end{proof}

    \begin{proof}[Proof of Theorem \ref{T:SUSP_EX}]
        Item (1) follows from Proposition \ref{P:FRAMING_TRIV} and Lemma \ref{L:NULL-HTPIC} and items (2) and (3) follow from Proposition \ref{P:SUSP_SPH_BDL}.
    \end{proof}


\section{Proof of Theorems \ref{T:S1_BUNDLES}, \ref{T:FREE_CIRCLE} and \ref{T:COHOM_4}, and Corollaries \ref{C:FREE_TORUS}, \ref{C:COHOM_2} and \ref{C:RIC>0}} 
\label{S:CIRCLE_BUNDLES}

    In this section, we prove Theorems \ref{T:S1_BUNDLES}, \ref{T:FREE_CIRCLE} and \ref{T:COHOM_4}, and Corollaries \ref{C:FREE_TORUS}, \ref{C:COHOM_2} and \ref{C:RIC>0}. First, for the proof of Theorem \ref{T:S1_BUNDLES}, we show the following more general result.
    \begin{theorem}
        \label{T:S1_BUNDLES2}
        Let $B^n=B_1\# B_2$, $n\geq 5$, and let $P\xrightarrow{\pi}B$ be a principal $S^1$-bundle with primitive Euler class $e$. We assume that $B_1$ is of the form \eqref{EQ:COND} and that $B_2$ is closed and simply-connected. Denote by $e_i$ the restriction of $e$ to $B_i$ and by $d_i$ the divisibility of $e_i$. If $b_2(B_1)=1$, we additionally assume that $d_1\equiv\pm 1\mod d_2$. Then, we have
        \[P\cong\begin{cases}
            \hat{B}_1\# \Sigma_{e_2}B_2,\quad &\text{if }B_1\text{ is non-spin},\\
            \hat{B}_1\# \widetilde{\Sigma}_{e_2}B_2,\quad &\text{if }B_1\text{ is spin},
        \end{cases} \]
        where $\hat{B}_1$ is of the form \eqref{EQ:COND} with
        \[ b_i(\hat{B}_1)=\begin{cases}
            b_2(B_1)-1,\quad & i=2,n-2,\\
            b_{i-1}(B_1)+b_{i}(B_1),\quad & 2<i<n-2,
        \end{cases} \]
        and $\hat{B}_1$ is spin if and only if the restriction of $e_1$ to each $(S^2\ttimes S^{n-2})$-summand in $B_1$ has odd divisibility.
    \end{theorem}
    Theorem \ref{T:S1_BUNDLES} now follows from Theorem \ref{T:S1_BUNDLES2} by setting $B_2=S^n$, in which case $\Sigma_0 B_2\cong \widetilde{\Sigma}_0 B_2\cong S^{n+1}$.

    Before we prove Theorem \ref{T:S1_BUNDLES2}, we first note the following observation.
    \begin{lemma}
        \label{L:COND_DIFFEO}
        A manifold $M$ of the form \eqref{EQ:COND} is uniquely determined (up to diffeomorphism) by its dimension $n$, the Betti numbers $b_2(M),\dots,b_{\lfloor\frac{n}{2}\rfloor}$, and  whether $M$ is spin or not. Conversely, any sequence $b_2,\dots,b_{\lfloor\frac{n}{2}\rfloor}\in\N_0$ with $b_{\frac{n}{2}}$ even if $n$ is even can be realized as the Betti numbers of an $n$-dimensional spin manifold of the form \eqref{EQ:COND}, and of an $n$-dimensional non-spin manifold of the form \eqref{EQ:COND} provided $b_2\geq 1$.
    \end{lemma}
    \begin{proof}
        Since
        \[(S^2\ttimes S^{n-2})\# (S^2\times S^{n-2})\cong (S^2\ttimes S^{n-2})\# (S^2\ttimes S^{n-2}) \]
        by Corollary \ref{C:CONN_SUM_DIFF}, the information whether $M$ is spin or not is sufficient (together with the Betti numbers) to determine its diffeomorphism type. All other claims are obvious.
    \end{proof}
    \begin{proof}[Proof of Theorem \ref{T:S1_BUNDLES2}]
        Let $\xi_0\in H_2(B_2)$ be a class such that $e_2\frown \xi_0=d_2$.
        We first assume that $b_2(B_1)\geq 2$. Let $k,l\in\Z$ so that $k d_1+l d_2=1$ (which exist since $e$ is primitive). Denote by $x_i^*$ a generator of the second cohomology of the $i$-th summand in $B_1$ that is a sphere bundle over $S^2$. Then, by Corollary \ref{C:HOMOLOGYDIFF}, we can apply a self-diffeomorphism of $B_1$ so that $e_1$ is given by $k d_1 x_1^*+d_1 x_2^*$. Hence, if we write $B_1\# B_2$ as $M\# N$, where $N$ is the summand of $B_1$ with $H^2(N)$ generated by $x_1^*$ and $M$ is the connected sum of all remaining summands, we can apply Corollary \ref{C:HOMOLOGY_DIFF_CS} or \ref{C:HOMOLOGY_DIFF_CS_NONSP} with $\xi=l\xi_0$ (and note that the class $l\xi_0$ can be represented by a map $S^2\to B_2$ by the Hurewicz Theorem). Thus, we obtain a self-diffeomorphism of $B_1\# B_2$ that maps $e$ to $x_1^*+d_1 x_2^*+e_2$.
        Hence, the restriction of $e$ to one $(S^2\times S^{n-2})$ or $(S^2\ttimes S^{n-2})$-summand is primitive.

        In case $b_2(B_1)=1$, we obtain the same conclusion by applying Corollary \ref{C:HOMOLOGY_DIFF_CS} or \ref{C:HOMOLOGY_DIFF_CS_NONSP} to $\xi=l\xi_0$, where here $l\in \Z$ is chosen so that $d_1+l d_2=\pm1$.

        We now repeatedly apply Theorem \ref{T:MAIN} to obtain that $P$ is the connected sum of $(S^3\times S^{n-2})$ (which is the total space of the principal $S^1$-bundle over $(S^2\times S^{n-2})$ or $(S^2\ttimes S^{n-2})$ with Euler class $x_1^*$) and twisted suspensions of $B_2$ along $e_2$ and of the remaining products of spheres or $(S^2\ttimes S^{n-2})$-summands. Thus, the claim now follows from Theorem \ref{T:SUSP_EX}.
    \end{proof}

    \begin{remark}
        The proof shows that Theorem \ref{T:S1_BUNDLES2} can be generalized to the case where $B_2$ is not simply-connected if we assume that there exists a homology class $\xi\in H_2(B_2)$ with $e_2\frown\xi=ld_2$ that is represented by a map $S^2\to B_2$. In this case $e_2$ might change within its equivalence class in $H^2_F(B_2)$. This can be avoided if one assumes that $H_1(B_2)$ is torsion-free, so that $H^2_F(B_2)\cong H^2(B_2)$.
    \end{remark}

    \begin{proof}[Proof of Corollary \ref{C:FREE_TORUS}]
        We first consider the case $k=1$ and assume that $M$ is of the form \eqref{EQ:COND} and its partial Euler characteristics satisfy the stated conditions. We define the manifold $B$, such that $B$ is of the form \eqref{EQ:COND}, has dimension $n-1$, is non-spin, and has Betti numbers $b_i(B)=(-1)^i\chi_i(M)$ for all $i=2,\dots,\lfloor\frac{n-1}{2}\rfloor$. Note that, if $n-1$ is even, then $b_{\frac{n-1}{2}}(B)$ is even by assumption, so $B$ is well-defined and unique by Lemma \ref{L:COND_DIFFEO}.

        Now, let $e\in H^2(B)$ be a primitive element that satisfies $e\equiv w_2(B)\mod 2$ if and only if $M$ is spin. We define $P$ as the total space of the principal $S^1$-bundle over $B$ with Euler class $e$. By Theorem \ref{T:S1_BUNDLES}, the manifold $P$ is of the form \eqref{EQ:COND} and satisfies the following conditions:
        \begin{itemize}
            \item $b_2(P)=b_2(B)-1=\chi_2(M)-1=b_2(M)$.
            \item For $2<i<\lfloor \frac{n-1}{2}\rfloor$, or $i=\frac{n-1}{2}$ if $n-1$ is even, we have
            \[b_i(P)=b_{i-1}(B)+b_i(B)=(-1)^{i-1}\chi_{i-1}(M)+(-1)^i\chi_i(M)=b_i(M).\] 
            \item If $n$ is even, we have
            \[b_{\frac{n}{2}}(P)=2b_{\frac{n-2}{2}}(B)=2(-1)^{\frac{n}{2}-1}\chi_{\frac{n}{2}-1}(M)=b_{\frac{n}{2}}(M),\]
            since $0=\chi(M)=2\chi_{\frac{n}{2}-1}(M)+(-1)^{\frac{n}{2}}b_{\frac{n}{2}}(M)$.
        \end{itemize}
        Since $P$ is spin if and only if $M$ is spin by Lemma \ref{L:PR_BDL_TOPOLOGY}, it follows that $P$ is diffeomorphic to $M$ by Lemma \ref{L:COND_DIFFEO}. Hence, $M$ admits a free $S^1$-action with quotient of the form \eqref{EQ:COND}.

        For general $k$ we iterate the above argument to obtain a sequence $M\cong P_k\xrightarrow{\pi_k}\dots\xrightarrow{\pi_1} P_0$ of principal $S^1$-bundles with Euler classes $e(\pi_i)\in H^2(P_{i-1})$, so that each $P_i$ is of the form \eqref{EQ:COND}. Let $e_i\in H^2(P_0)$ so that $e(\pi_i)=\pi_{i-1}^*\dots\pi_1^* e_i$, which exists since each $e(\pi_i)$ is primitive and the induced map on $H^2$ of each $\pi_i$ can be identified with the quotient map by $e_i$, by Lemma \ref{L:PR_BDL_TOPOLOGY}. Define $P$ as the principal $T^k$-bundle with Euler class $(e_1,\dots,e_k)$. Then, by Lemma \ref{L:TORUS_BDL_STR}, $P$ is diffeomorphic to $P_k\cong M$, showing that $M$ admits a free $T^k$-action with quotient $P_0$, which is of the form \eqref{EQ:COND}.

        Conversely, assume that $M$ admits a free $S^1$-action with quotient of the form \eqref{EQ:COND}. Then $M$ is the total space of a principal $S^1$-bundle with base $B$ of the form \eqref{EQ:COND} and it follows inductively from Theorem \ref{T:S1_BUNDLES} that $b_i(B)=(-1)^i\chi_i(M)$ for $i=2,\dots,\lfloor\frac{n-1}{2}\rfloor$, which is non-negative since $b_i(B)\geq 0$. Further, by Lemma \ref{L:COND_DIFFEO}, we have that $(-1)^{\frac{n-1}{2}}\chi_{\frac{n-1}{2}}(M)$ is even when $n-1$ is even, and, since $b_{\frac{n}{2}}(M)=2b_{\frac{n}{2}-1}(B)=(-1)^{\frac{n}{2}-1}\chi_{\frac{n}{2}-1}(M)$ if $n$ is even, it also follows that $\chi_n(M)=0$ if $n$ is even. The statement for general $k$ now follows by induction.      
    \end{proof}

    To prove Theorem \ref{T:FREE_CIRCLE}, we first prove the following lemmas.
    \begin{lemma}
        \label{L:PROJ_BUNDLE}
        Let $E\xrightarrow{\xi}S^2$ be a complex vector bundle of rank $r+1$ and let $P(E)\to S^2$ be the associated projective bundle, i.e.\ $P(E)$ consists of all complex one-dimensional subspaces of fibers in $E$, so we obtain a fiber bundle with fiber $\C P^r$. Let $P\to P(E)$ denote the sphere bundle of the tautological line bundle over $P(E)$. Then
        \[ P\cong \begin{cases}
            S^2\times S^{2r+1},\quad & \text{ if } c_1(\xi)\text{ is even},\\
            S^2\ttimes S^{2r+1},\quad & \text{ if } c_1(\xi)\text{ is odd}.
        \end{cases} \]
    \end{lemma}
    \begin{proof}
        By definition, the total space of the sphere bundle $S(T)\to P(E)$ of the tautological line bundle $T\to P(E)$ is given by
        \[ S(T)=\{(v,\varphi)\in E\times P(E)\mid v\in\varphi,\, \lVert v\rVert=1\}.  \]
        By projection onto the first coordinate, we obtain an identification of $S(T)$ with the total space $S(E)$ of the sphere bundle of $\xi$. Since $w_2(\xi)=c_1(\xi)\mod 2$, the claim follows.
    \end{proof}
    \begin{lemma}
        \label{L:CPm_BUNDLE}
        There exists a linear $S^r$-bundle $E\to\C P^m$, $r\geq 2$, with the following properties:
        \begin{enumerate}
            \item The total space $E$ is spin if and only if $m$ is even.
            \item If $P\to E$ denotes the principal $S^1$-bundle, whose Euler class is given by the pull-back of a generator of $H^2(\C P^m)$, then $P\cong S^{2m+1}\times S^r$.
        \end{enumerate}
    \end{lemma}
    \begin{proof}
        We define $\overline{E}\xrightarrow{\pi} \C P^m$ as the sum of the tautological line bundle with the trivial bundle $\underline{\R}^{r-1}_{\C P^m}$. Then $w_2(\pi)$ is non-trivial (see e.g.\ \cite[Theorem 14.4]{MS74}). If $E\to \C P^m$ denotes the corresponding sphere bundle, we have $TE\oplus\underline{\R}_E\cong \pi^*T\C P^m\oplus \pi^*\overline{E}$, cf.\ Lemma \ref{L:SPIN}. Hence, $w_2(E)$ is trivial if and only if $w_2(\C P^m)$ is non-trivial, which is the case if and only if $m$ is even.

        By construction, the bundle $P\to E$ fits into the following pull-back diagram:
        \[
            \begin{tikzcd}
                P\arrow{r}\arrow{d} &S^{2m+1}\arrow{d}\\
                E\arrow{r} & \C P^m
            \end{tikzcd}
        \]
        Here $S^{2m+1}\to\C P^m$ denotes the Hopf fibration (i.e.\ the principal $S^1$-bundle whose Euler class is a generator of $H^2(\C P^m)$). It follows that $P\to S^{2m+1}$ is a linear $S^r$-bundle. Since the structure group of this bundle is contained in $\mathrm{SO}(2)\cong S^1$, and since $S^1$ has trivial higher homotopy groups, this bundle is trivial, so $P\cong S^{2m+1}\times S^r$.
    \end{proof}
   
    \begin{proof}[Proof of Theorem \ref{T:FREE_CIRCLE}]
        Let $E(m,r)$ denote the total space of the linear $S^r$-bundle over $\C P^m$ from Lemma \ref{L:CPm_BUNDLE}. We then set
        \[ E_m^r=\begin{cases}
            \C P^m\times S^r,\quad & m\text{ odd},\\
            E(m,r),\quad & m\text{ even},
        \end{cases}\quad \text{ and }\widetilde{E}_m^r=\begin{cases}
            E(m,r),\quad & m\text{ even},\\
            \C P^m\times S^r,\quad & m\text{ odd},
        \end{cases} \]
        so that $E_m^r$ is spin and $\widetilde{E}_m^r$ is non-spin, and the principal $S^1$-bundle over $E_m^r$ or $\widetilde{E}_m^r$ whose Euler class is the pull-back of a generator of $H^2(\C P^m)$ has total space $S^{2m+1}\times S^r$.
        
        \emph{(1).} We define
        \[B=\begin{cases}
            \left(\#_l\C P^{\frac{n-1}{2}}\right)\#\left(\#_{i=1}^{\frac{n-3}{2}}\#_{b_{2i+1}(M)}E_i^{n-2i-1}\right),\quad & n\equiv 3\mod 4,\\
            \left(\#_l\C P^{\frac{n-1}{2}}\right)\#\left(\#_{i=1}^{\frac{n-3}{2}}\#_{b_{2i+1}(M)}\widetilde{E}_i^{n-2i-1}\right),\quad & n\equiv 1\mod 4.\\
        \end{cases} \]
        for some $l\geq 0$. Let $e\in H^2(B)$ be a class that restricts to a generator of $H^2(\C P^{\frac{n-1}{2}})$ on each $\C P^{\frac{n-1}{2}}$-summand and to the pull-back of a generator of $H^2(\C P^i)$ on each $E_i^j$ and $\widetilde{E}_i^j$-summand. Then, for the principal $S^1$-bundle $P\to B$ with Euler class $e$, we have by Theorems \ref{T:MAIN} and \ref{T:SUSP_EX} and Lemma \ref{L:CPm_BUNDLE} (note that the summands of $B$ are either all spin or all non-spin)
        \[ P\cong \#_a (S^2\times S^{n-2})\#_{i=1}^{\frac{n-3}{2}}\#_{b_{2i+1(M)}}(S^{2i+1}\times S^{n-2i-1}), \]
        where $a=l-1+\sum_{i=1}^{\frac{n-5}{2}} b_{2i+1}(M)$. Thus, since $b_{2i+1}(M)= b_{n-2i-1}(M)$, we have for $3\leq i\leq \frac{n-1}{2}$, that $b_i(P)=b_i(M)$. Hence, if $M$ is spin, it becomes diffeomorphic to $P$ after connected sum with sufficiently many copies of $(S^2\times S^{n-2})$ and by choosing $l$ large enough.

        For the non-spin case, or the case where $M$ is spin and we take connected sums with copies of $(S^2\ttimes S^{n-2})$, we replace one $E_i^j$-summand in $B$ by $\widetilde{E}_i^j$ or vice versa, provided there is a non-trivial summand of this form. Then $P$ has a summand of the form $(S^2\ttimes S^{n-2})$, hence the claim follows for $l$ large enough by Corollary \ref{C:CONN_SUM_DIFF}. If there exists no such summand, we additionally introduce a summand for $B$ given by $E_{\frac{n-3}{2}}^2\#\widetilde{E}_{\frac{n-3}{2}}^2$, which results in an additional summand for $P$ given by $\#_4(S^2\ttimes S^{n-2}) $.
        
        \emph{(2).} We consider each dimension separately. First, note that the result in dimension 5 was shown in \cite[Corollary 2]{DL05} by proving that any $5$-manifold of the form \eqref{EQ:COND} is the total space of a principal circle bundle over a closed, simply-connected 4-manifold. The $6$-dimensional case follows directly from Corollary \ref{C:FREE_TORUS} (see also \cite[Corollary B]{CG20} and \cite[Theorem C]{Duan}).
        
        For dimensions $7$--$10$, we summarize in Table \ref{TB:BUNDLES} how the base manifold $B$ in each case is given. One then easily verifies, using Theorems \ref{T:MAIN} and \ref{T:SUSP_EX}, that the total space of the principal circle bundle over $B$ with suitable Euler class $e$ is diffeomorphic to $M$. By $P(E)$ we denote the total space of a projective bundle of a vector bundle with odd first Chern class of appropriate dimension (cf.\ Lemma \ref{L:PROJ_BUNDLE}). The Euler class $e$ will always be the pull-back of a generator of $H^2(\C P^i)$ on each summand of the form $E_i^j$, the Euler class of the tautological line bundle over $P(E)$, and a generator of the second cohomology on each summand of the form $S^2\times S^{n-2}$ and $S^2\ttimes S^{n-2}$.

        \begin{table}[ht]
        \begin{center}
            \caption{Manifolds $M$ of the form \eqref{EQ:COND} and quotient manifold $B$ of a free circle action on $M$.}
            \label{TB:BUNDLES}
         {\renewcommand{\arraystretch}{1.2}
            \begin{tabular}{|c|c|c|c|c|c| l|l |} \hline
                \multicolumn{7}{|l|}{Manifold $M$, $n=\dim(M)$} & Base manifold $B$ \\\cline{1-7}
                $n$ & $w_2$ & $b_2$ & $b_3$ & $b_4$ & $b_5$ & condition &\\
                \hline\hline 
                $7$ & $0$ & $p$ & $q$ & & & $q$ even & $\#_{p+1}\C P^3\#_{\frac{q}{2}}(S^3\times S^3)$ \\
                & $0$ & $p$ & $q$ & & & $q$ odd & $(S^2\times S^4)\#_{p}\C P^3\#_{\frac{q-1}{2}}(S^3\times S^3)$\\
                & $1$ & $p$ & $q$ & & & $q\geq 2$ even, $p\geq 1$ & $(S^2\ttimes S^4)\#(S^2\times S^4)\#_{p-1}\C P^3\#_{\frac{q-2}{2}}(S^3\times S^3)$ \\
                & $1$ & $p$ & $q$ & & & $q$ odd, $p\geq 1$ & $(S^2\ttimes S^4)\#_{p}\C P^3\#_{\frac{q-1}{2}}(S^3\times S^3)$\\
                & $1$ & $p$ & $0$ & & & $p>1$ & $\widetilde{E}_2^2\#_{p-1}\C P^3$\\
                & $1$ & $1$ & $0$ & & &  & $P(E)$\\\hline
                8 & $0$ & $p$ & $q$ & $2r$ & & $p+r+1=q$ & $\#_{p+1}(S^2\times S^5)\#_r(S^3\times S^4)$\\
                & $1$ & $p$ & $q$ & $2r$ & & $p+r+1=q$ & $(S^2\ttimes S^5)\#_{p}(S^2\ttimes S^5)\#_r(S^3\times S^4)$\\\hline
                9 & $0$ & $p$ & $q$ & $r$ & & $q>0$,  & $\#_{p+1-a}\C P^4\#_a (S^2\ttimes S^6)\#_b(S^3\times S^5)\#_{\frac{r-b}{2}} (S^4\times S^4)$\\ & & & & & & $1+p+r\geq q$ & for $a+b=q$, $a\leq p+1$, $b\leq r$, $r-b$ even\\
                & $0$ & $q$ & $0$ & $r$ & & $r$ even & $\#_{p+1}\C P^4\#_{\frac{r}{2}}(S^4\times S^4) $\\
                & $0$ & $p$ & $0$ & $r$ & & $r$ odd & $\widetilde{E}_2^4\#_{p}\C P^4\#_{\frac{r-1}{2}}(S^4\times S^4) $\\
                & $1$ & $p$ & $q$ & $r$ & & $p>0$, $q>1$, & $\#_{p+1-a}\C P^4\#_a (S^2\times S^6)\#_b(S^3\times S^5)\#_{\frac{r-b}{2}} (S^4\times S^4)$\\ & & & & & & $1+p+r\geq q$ & for $a+b=q$, $1\leq a\leq p+1$, $b\leq r$, $r-b$ even\\
                & $1$ & $p$ & $1$ & $r$ & & $p>0$, $r$ even & $\#_{p}\C P^4\# (S^2\times S^6)\#_{\frac{r}{2}} (S^4\times S^4)$\\ 
                & $1$ & $p$ & $1$ & $r$ & & $p>0$, $r$ odd & $\#_{p-1}\C P^4\#(S^2\ttimes S^6)\#E_2^4\#_{\frac{r-1}{2}} (S^4\times S^4)$\\
                & $1$ & $p$ & $0$ & $r$ & & $p>0$, $r\geq 2$ even &  $E_2^4\# \widetilde{E}_2^4\#_{p-1}\C P^4\#_{\frac{r-2}{2}}(S^4\times S^4) $\\
                & $1$ & $p$ & $0$ & $r$ & & $p>0$, $r$ odd &  $E_2^4\#_{p}\C P^4\#_{\frac{r-1}{2}}(S^4\times S^4) $\\
                & $1$ & $p$ & $0$ & $0$ & & $p>1$ &  $E_3^2\#_{p-1}\C P^4$\\
                & $1$ & $1$ & $0$ & $0$ & & & $P(E)$\\\hline
                $10$ & $0$ & $p$ & $q$ & $r$ & $2s$ & $p+r+1=q+s$, & $\#_q (S^2\times S^7)\#_r (S^4\times S^5)\#_{s-r}E_2^5 $ \\
                & & & & & & $s\geq r$ & \\
                & $0$ & $p$ & $q$ & $r$ & $2s$ & $p+r+1=q+s$, & $\#_{p+1} (S^2\times S^7)\#_{r-s}(S^3\times S^6)\#_s (S^4\times S^5)$ \\
                & & & & & & $s< r$ & \\
                & $1$ & $p$ & $q$ & $r$ & $2s$ & $p+r+1=q+s$, & $(S^2\ttimes S^7)\#_{q-1} (S^2\ttimes S^7)\#_r (S^4\times S^5)\#_{s-r}E_2^5 $ \\
                & & & & & & $s\geq r$, $p,q>0$ & \\
                & $1$ & $p$ & $q$ & $r$ & $2s$ & $p+r+1=q+s$, & $(S^2\times S^7)\#_{p} (S^2\ttimes S^7)\#_{r-s}(S^3\times S^6)\#_s (S^4\times S^5)$ \\
                & & & & & & $s< r$, $p>0$ & \\
                & $1$ & $p$ & $0$ & $r$ & $2s$ & $p+r+1=s$, & $\#_r (S^4\times S^5)\#\widetilde{E}_2^5\#_{s-r-1}E_2^5 $ \\
                & & & & & & $p>0$ & \\\hline
            \end{tabular}
            }
        \end{center}
        \end{table}
    \end{proof}

    We now show that the additional assumption in the $9$-dimensional case in Theorem \ref{T:FREE_CIRCLE} cannot be removed in general.
    \begin{proposition}
        \label{P:S3xS6}
        The manifold $\#_{2p+1} (S^3\times S^6)$ does not admit a free circle action for any $p>0$.
    \end{proposition}
    \begin{proof}
        Suppose there exists a principal $S^1$-bundle $P\xrightarrow{\pi} B$ with $P\cong \#_{2p+1}(S^3\times S^6)$. Then $B$ is a closed $8$-manifold, and, by the long exact sequence of homotopy groups for the bundle $\pi$, the manifold $B$ is simply-connected. Then, by the Gysin sequence, cup product with the Euler class $\cdot\smile e(\pi)\colon H^i(B)\to H^{i+2}(B)$ is an isomorphism for $i=0,3,6$, injective for $i=4$ and surjective for $i=2$. In particular, $H^2(B)\cong H^6(B)\cong\Z$ and $H^4(B)$ is either trivial or isomorphic to $\Z$, in particular torsion-free. By using Poincaré duality and the universal coefficient theorem, it follows that $B$ has torsion-free cohomology.

        From the Gysin sequence we can now extract the following exact sequence:
        \[ 0\longrightarrow H^3(B)\xrightarrow{\pi^*}H^3(P)\longrightarrow H^2(B)\xrightarrow{\cdot\smile e(\pi)}H^4(B)\longrightarrow0. \]
        It follows that, depending on whether $H^4(B)$ is trivial or isomorphic to $\Z$, $H^3(B)$ is isomorphic to $\Z^{2p}$ or $\Z^{2p+1}$. We now show that only the latter can be the case.

        For that, let $x\in H^3(B)$ and $y\in H^5(B)$ with $x\smile y\neq 0$, which exist by Poincaré duality (here, we use $p>0$). Since $\cdot\smile e(\pi)\colon H^3(B)\to H^5(B)$ is an isomorphism, there exists $y'\in H^3(B)$ with $y=y'\smile e(\pi)$. In particular, $x\smile y'\neq 0$. Since $\pi^*x\smile \pi^*y'=0$ (as $P$ has trivial cup products in degree $3$), by exactness of the Gysin sequence, there exists $z\in H^4(B)$ with $z\smile e(\pi)=x\smile y'\neq 0$. In particular, $H^4(B)$ is non-trivial, so $H^4(B)\cong\Z$ and $H^3(B)\cong \Z^{2p+1}$.

        By Poincaré duality and since $\cdot\smile e(\pi)\colon H^6(B)\to H^8(B)$ is an isomorphism, the cup product $H^3(B)\times H^3(B)\to H^6(B)\cong \Z$ is a non-degenerate skew-symmetric bilinear form. In particular, $H^3(B)$ has even rank, which is a contradiction.
    \end{proof}

    Theorem \ref{T:COHOM_4} is a direct consequence of the following theorem. Recall that we define $a_{ki}(r)$ for $r,k\in\N_0$ and $2\leq i \leq k+2$ by
    \[a_{ki}(r)= (i-2)
    \binom{k}{i-1}
    +r
    \binom{k}{i-2}
    +(2+k-i)
    \binom{k}{i-3}. 
    \]
    \begin{theorem}
        \label{T:Tk_BUNDLES}
		Let $P$ be the total space of a principal $T^k$-bundle over a closed, simply-connected 4-manifold $B$ and denote by $e(\pi)=(e_1(\pi),\dots,e_k(\pi))\in H^2(B)^k$ its Euler class. If $P$ is simply-connected, or, equivalently, $e(\pi)$ can be extended to a basis of $H^2(B)$, then $P$ is of the form \eqref{EQ:COND} with $b_i(P)=a_{ki}(b_2(B)-k)$ and $P$ is spin if and only if $w_2(B)$ is contained in the subspace of $H^2(B,\Z/2)$ generated by $e(\pi)\mod 2$.
    \end{theorem}

    \begin{proof}
        The claims on simply-connectedness and the spin condition follow from Lemmas \ref{L:PR_BDL_TOPOLOGY} and \ref{L:PR_BDL_CHAR_CL}. By Lemma \ref{L:TORUS_BDL_STR}, the bundle $\pi$ can be decomposed into a sequence of principal $S^1$-bundles, which, by Lemma \ref{L:PR_BDL_TOPOLOGY}, all have simply-connected total space. We now proceed by induction.

        The case $k=1$ is a consequence of the classification of closed, simply-connected 5-manifolds by Smale \cite{Sm62} and Barden \cite{Ba65} and was treated by Duan and Liang \cite{DL05}. Now assume that $B^n$ is a manifold of the form \eqref{EQ:COND} with $b_i(B)=a_{ki}(r)$ for some $r\in\N$ and let $P\xrightarrow{\pi}B$ be a principal $S^1$-bundle with $P$ simply-connected. Then, by Theorem \ref{T:S1_BUNDLES}, the manifold $P$ is also of the form \eqref{EQ:COND} and we have
        \begin{align*}
            b_2(P)=b_2(B)-1=r-1=a_{k+1,2}(r-1)
        \end{align*}
        and
        \begin{align*}
            b_i(P)=&b_{i-1}(B)+b_i(B)=a_{k,i-1}(r)+a_{k,i}(r)\\
            =& (i-3)\binom{k}{i-2}
            +r\binom{k}{i-3}
            +(3+k-i)\binom{k}{i-4}\\
            &+(i-2)\binom{k}{i-1}
            +r\binom{k}{i-2}
            +(2+k-i)\binom{k}{i-3}\\
            =& (i-2)\binom{k+1}{i-1}
            -\binom{k}{i-2}
            +r\binom{k+1}{i-2}
            +(3+k-i)\binom{k+1}{i-3}
            -\binom{k}{i-3}\\
            =& (i-2)\binom{k+1}{i-1}
            +(r-1)\binom{k+1}{i-2}
            +(3+k-i)\binom{k+1}{i-3}
            =a_{k+1,i}(r-1).
        \end{align*}
        for $2<i<n-2$.
    \end{proof}
    \begin{proof}[Proof of Theorem \ref{T:COHOM_4}]
        If $M$ is a closed, simply-connected $n$-manifold with a free action of the torus $T^{n-4}$, then, by taking the quotient $B=M/T^{n-4}$, we obtain a principal $T^{n-4}$-bundle over the simply-connected $4$-manifold $B$ with total space $M$. Hence, we can apply Theorem \ref{T:Tk_BUNDLES}.

        Conversely, by Theorem \ref{T:Tk_BUNDLES} any $n$-manifold $M$ of the form \eqref{EQ:COND} with $b_i(M)=a_{n-4,i}(b_2(M))$ is the total space of a principal $T^{n-4}$-bundle over
        \[B=\#_{b_2(M)+n-4}\C P^2 \]
        (or any other closed, simply-connected non-spin 4-manifold $B$ with $b_2(B)=b_2(M)+n-4$) with Euler class $e\in H^2(B)^{n-4}$ that can be extended to a basis of $H^2(B)$ and so that $w_2(B)$ is contained in the subspace generated by $e\mod 2$ if and only if $M$ is spin.
    \end{proof}

    \begin{remark}
        \label{R:COHOM_4}
        Note that a closed, simply-connected $n$-manifold $M$ with $n\geq 4$ cannot admit a free action of a torus $T^k$ with $k>n-4$. To see this, assume that such an action exists. Then, by dividing out a subtorus of dimension $n-4$, we obtain a free action of $T^{n-4-k}$ on the simply-connected $4$-manifold $M/T^{n-4}$. However, a simply-connected 4-manifold has positive Euler characteristic, thus admitting no free torus action by Lemma \ref{L:FREE_ACTION_CHAR}.
    \end{remark}

    \begin{proof}[Proof of Corollary~\ref{C:COHOM_2}]
        First, suppose that such an action exists. By taking the quotient of $M$ by the free subaction of cohomogeneity 6, we obtain a closed, simply-connected $6$-manifold $M/T^{n-6}$ with an effective action of $T^4$. By the classification of Oh \cite{Oh82}, the manifold $M/T^{n-6}$ is of the form \eqref{EQ:COND} and the Betti numbers satisfy the assumptions of Theorem \ref{T:COHOM_4}. Hence, there exists a free $T^2$-action on $M/T^{n-6}$. By the lifting results of \cite{HY,Su}, $M$ therefore admits a free $T^{n-4}$-action, and the claim follows from Theorem \ref{T:COHOM_4}.

        Conversely, if $M$ is of the form \eqref{EQ:COND} with $b_i(M)=a_{ki}(b_2(M))$ for all $2\leq i\leq n-2$, then, by Theorem \ref{T:Tk_BUNDLES}, $M$ is the total space of a principal $T^{n-4}$-bundle over $B=\#_{b_2(M)+n-4}\C P^2$. By the classification of closed, simply-connected 4-manifolds with an effective $T^2$-action by Orlik and Raymond \cite{OR70}, $B$ admits an effective $T^2$-action. Hence, by the lifting results of \cite{HY,Su}, $M$ admits a cohomogeneity-$2$ torus action that contains a free subaction of cohomogeneity $4$, in particular it contains a free subaction of cohomogeneity $6$.
    \end{proof}
    We note that it follows from the proof of Corollary \ref{C:COHOM_2} that, if $M$ admits a cohomogeneity-two torus action that contains a free subaction of cohomogeneity six, then $M$ also admits a (possibly different) cohomogeneity-two torus action with a free subaction of cohomogeneity four. 
    
    \begin{remark}\label{R:NO.FREE.SUBTORUS}
        Note that not all cohomogeneity-two actions of $T^{n-2}$ on a closed, simply-connected $n$-manifold $M$ admit a free subaction of cohomogeneity six. Indeed, if every involution of $T^{n-2}$ is contained in one of the isotropy subgroups of the action, every $T^1$-subgroup of $T^{n-2}$ necessarily intersects non-trivially with an isotropy subgroup. Such an action can for example be constructed as follows: 

        Let $A=\{0,1\}^{n-2}\setminus\{0\}$ and consider a $(2^{n-2}-1)$-gon, where each edge is labeled by one of the vectors in $A$ so that each element of $A$ appears precisely once. It is easily verified that this is a legally weighted orbit space in the sense of \cite[Section 2]{GGK14}, and therefore defines closed, simply-connected $n$-manifold $M$ with a cohomogeneity-two torus action for which $T^1(v)$ appears as an isotropy subgroup for all $v\in A$, where $T^1(v)$ is the circle in $T^{n-2}$ with slope $v$. Hence, by construction, all involutions of $T^{n-2}$ are contained in an isotropy subgroup. We thank Lee Kennard and Lawrence Mouillé for providing this example.
    \end{remark}

    \begin{proof}[Proof of Corollary~\ref{C:RIC>0}]
        We use the core metric construction introduced by Burdick \cite{Bu19} to construct a metric of positive Ricci curvature on each quotient manifold. By \cite[Theorem C]{Bu19}, \cite[Theorem B]{Bu20} and \cite[Theorem C]{Re21}, spheres, complex projective spaces and total spaces of linear sphere bundles over spheres and complex projective spaces admit core metrics, where in the latter case the dimension is at least 6. Hence, by \cite[Theorem B]{Bu19}, any finite connected sum of such manifolds admits a metric of positive Ricci curvature. In dimension 5, it was shown by Sha and Yang \cite[Theorem 1]{SY91}, that any 5-manifold of the form \eqref{EQ:COND} admits a metric of positive Ricci curvature. Finally, by a classical result of Nash \cite[Theorem 3.5]{Na79}, projective bundles over spheres admit metrics of positive Ricci curvature.

        Hence, for each manifold $M$ appearing in Corollaries \ref{C:FREE_TORUS} and \ref{C:COHOM_2} and in Theorems \ref{T:FREE_CIRCLE} and \ref{T:COHOM_4}, and for the free torus action considered in the proof of the corresponding result, the quotient admits a metric of positive Ricci curvature. Hence, $M$ is the total space of a principal torus bundle over a manifold with a metric of positive Ricci curvature. Since $M$ is simply-connected, it follows from the lifting result of Gilkey--Park--Tuschmann \cite{GPT98}, that $M$ admits a metric of positive Ricci curvature that is invariant under the corresponding torus action.
    \end{proof}
    
\begin{conflict}
    The authors have no competing interest to declare.
\end{conflict}

\begin{funding}
    Both authors acknowledge funding by the Deutsche Forschungsgemeinschaft (DFG, German Research Foundation) -- 281869850 (RTG 2229) and grant GA 2050 2-1 within the SPP 2026 ``Geometry at Infinity''. Further, P.R. acknowledges funding by the SNSF-Project 200020E\textunderscore 193062 and the DFG-Priority programme SPP 2026.
\end{funding}
 
\bibliographystyle{plainurl}
\bibliography{References}
\end{document}